\documentclass[reqno, 12pt]{amsart}
\RequirePackage{etex}

\makeatletter
\let\origsection=\section \def\section{\@ifstar{\origsection*}{\mysection}} 
\def\mysection{\@startsection{section}{1}\z@{.7\linespacing\@plus\linespacing}{.5\linespacing}{\normalfont\scshape\centering\S}}
\makeatother    

\usepackage{geometry}
\numberwithin{equation}{section}
\numberwithin{figure}{section}

\usepackage{pdfpages}
\usepackage{pifont}%for scissors
\usepackage{xcolor}
\usepackage{hyperref}
\hypersetup{
    unicode,
    colorlinks,
    linkcolor={red!60!black},
    citecolor={green!60!black},
    urlcolor={blue!60!black}
}

\newcommand{\arc}[2]{$#1$-$#2$ cutline}
\usepackage{amsfonts,amsthm,amssymb,amsmath}
\usepackage[T1]{fontenc}
\usepackage{microtype}
\usepackage{graphicx}
\usepackage[usenames,dvipsnames]{pstricks}
\usepackage{epsfig}
\usepackage{verbatim}
\usepackage{scalerel}
\usepackage{tikz}
\usepackage{tkz-berge}
\usepackage{tkz-euclide}
\usepackage{pgfplots}
\usepackage[outline]{contour}
\usetikzlibrary{hobby}
\usetikzlibrary{decorations.pathmorphing}
\usetikzlibrary{decorations.markings}
\usetikzlibrary{pgfplots.fillbetween}
\usetikzlibrary{decorations.pathmorphing, arrows, calc, shapes, matrix, arrows.meta}

\definecolor{myred}{rgb}{.9,.1,.2}

\usepackage{enumerate,enumitem}
\usepackage{thmtools, thm-restate}%For duplicating theorem numbers.
\usetikzlibrary{decorations.pathmorphing, arrows, calc, matrix, arrows.meta}

\usepackage{mathtools}
\usepackage{subcaption}

\newtheorem{theorem}{Theorem}[section]
\numberwithin{theorem}{section}
\newtheorem{lemma}[theorem]{Lemma}  
\newtheorem{cor}[theorem]{Corollary}

\newtheorem{conj}[theorem]{Conjecture}

\theoremstyle{definition}
\newtheorem{defn}[theorem]{Definition}

\theoremstyle{remark}

\renewcommand{\triangleleft}{\vartriangleleft}
\renewcommand{\leq}{\leqslant}
\renewcommand{\geq}{\geqslant}

\renewcommand{\rho}{\varrho}
\newcommand{\nottriangleleft}{\not\kern-1pt\mathrel{\triangleleft}}
\newcommand{\cut}{\ensuremath{\mathrm{Cut}}}

\newcommand{\id}{\ensuremath{\mathrm{id}}}
\newcommand{\Int}[1]{\mathring{#1}}
\newcommand{\cutline}{cutline}

\pgfplotsset{compat=1.14}

\begin{document}
\title{Bounding the cop number of a graph by its genus}   
\author{Nathan Bowler}
\address{Nathan Bowler, University of Hamburg, Department of Mathematics, Bundesstra{\ss}e 55, 20146 Hamburg, Germany}
\email{nathan.bowler@uni-hamburg.de}
\author{Joshua Erde}
\address{Joshua Erde, Graz University of Technology, Institute of Discrete Mathematics, Steyrergasse 30, 8010 Graz, Austria}
\email{erde@math.tugraz.at}
\author{Florian Lehner}
\thanks{Florian Lehner acknowledges the support of the Austrian Science Fund (FWF), grant J 3850-N32}
\address{Florian Lehner, Graz University of Technology, Institute of Discrete Mathematics, Steyrergasse 30, 8010 Graz, Austria}
\email{florian.lehner@tugraz.at}
\author{Max Pitz}
\address{Max Pitz, University of Hamburg, Department of Mathematics, Bundesstra{\ss}e 55, 20146 Hamburg, Germany}
\email{max.pitz@uni-hamburg.de}

\begin{abstract}
It is known that the cop number $c(G)$ of a connected graph $G$ can be bounded as a function of the genus of the graph $g(G)$. The best known bound, that $c(G) \leq \left\lfloor \frac{3 g(G)}{2}\right\rfloor + 3$, was given by Schr\"{o}der, who conjectured that in fact $c(G) \leq g(G) + 3$. We give the first improvement to Schr\"{o}der's bound, showing that $c(G) \leq \frac{4g(G)}{3}  + \frac{10}{3}$.
\end{abstract}

\maketitle

%NOTE FOR WHOEVER SUBMITS THIS: We got an email from Liu Mingrui asking for this preprint, and we promised to send it to them when it is done. So please send the paper to lmr16@mails.tsinghua.edu.cn
%(JOSH) : Also, Imre asked me to send him a copy (just so I remember)
%(Josh) : Also peter_bradshaw@sfu.ca e-mailed asking for a pre-print

\section{Introduction}
The game of \emph{cops and robbers} was introduced independently by Nowakowski and Winkler \cite{NW83} and Quillot \cite{Q78}.
%The Quillot reference varies sometimes in the literature, this is the one Quillot uses himself. 
The game is a pursuit game played on a connected graph $G=(V,E)$ by two players, one player controlling a set of $k \geq 1$ \emph{cops} and the other controlling a \emph{robber}. Initially, the first player chooses a starting configuration $(c_1,c_2,\ldots,c_k) \in V^k$ for the cops and then the second player chooses a starting vertex $r \in V$ for the robber. The game then consists of alternating moves, one move by the cops and then a subsequent move by the robber. For a cop move, each cop may move to a vertex adjacent to his current location, or stay still, and the same goes for a subsequent move of the robber. Note that each cop may change his position in a move, and that multiple cops may occupy the same vertex. The first player wins if at some time there is a cop on the same vertex as the robber, otherwise the robber wins. We define the \emph{cop number} $c(G)$ of a graph $G$ to be the smallest number of cops $k$ such that the first player has a winning strategy in this game. 

Bounding the cop number $c(G)$ in terms of invariants of the graph $G$ is a well-studied problem (see for example \cite{BN11}). Of particular interest is Meyniel's Conjecture that $c(G) = O(\sqrt{n})$ holds for every graph $G$ on $n$ vertices. Currently the best known bound is
\[
c(G) = O\left( \frac{n}{2^{(1-o(1))\sqrt{\log{n}}}} \right)
\]
proved independently by  Frieze, Krivelevich and Loh \cite{FKL12}, Lu and Peng \cite{LP12}, and Scott and Sudakov \cite{SS11}.

In another direction, the cop number has been studied with regard to topological properties of the graph (see \cite{BM17} for a recent survey). As an early example of this Aigner and Fromme \cite{AF84} showed that every planar graph has cop number at most three. For a generalisation of this result, consider graphs with bounded genus. Given a graph $G$ let us write $g(G)$ for the \emph{genus} of $G$, that is, the smallest $k$ such that $G$ can be drawn on an orientable surface of genus $k$ without crossing edges. For $g \in \mathbb{N}$ we define
\[
c(g) := \max \{ c(G) \mid g(G) = g \}.
\]

Using similar ideas to  \cite{AF84}, Quillot \cite{Q85} showed that $c(g) \leq 2 g +3$. These methods were refined by Schr\"{o}der to give the following bound, which is currently the best known.
\begin{theorem}[Schr\"{o}der \cite{S01}]\label{t:schr}
For every $g \in \mathbb N$, we have
\[
c(g) \leq \left\lfloor \frac{3 g}{2}\right\rfloor + 3.
\]
\end{theorem}
These proofs all hinge on the same basic idea that a single cop can `guard' a geodesic path in the graph $G$. Broadly, the strategy is to inductively find collections of geodesic paths (initially in $G$, but later in some subgraph, to which the cops have restricted the robber) such that, if we delete these paths, then each component has strictly smaller genus than before. After a fixed number of steps the robber is restricted to a planar graph, in which three further cops can catch him by the result of Aigner and Fromme.

In Quillot's proof, two cops are needed to restrict the robber to a subgraph whose genus is strictly smaller than that of the original graph. Schr\"{o}der's improvement was to find a strategy such that three cops could be used to reduce the genus by at least two. Perhaps the best we could hope for with a similar strategy would be to use a single cop to reduce the genus of the graph by one, and motivated by this Schr\"{o}der conjectured the following bound.

\begin{conj}[Schr\"{o}der \cite{S01}] For every $g \in \mathbb N$, we have
\[c(g) \leq g + 3.\]
\end{conj}

We improve Quillot's and Schr\"oder's proof idea by presenting a strategy in which either 4 cops can reduce the genus by at least 3, or 8 cops can reduce the genus by at least 6, or otherwise 12 cops can reduce the genus by at least 9. Correspondingly, our main result of this paper improves the bound in Theorem \ref{t:schr} as follows:

\begin{restatable}{theorem}{main}\label{t:main}
For every $g \in \mathbb N$, we have
\[
c(g) \leq \frac{4g}{3}   + \frac{10}{3}.
\]
\end{restatable}

In order to prove Theorem \ref{t:main} we will introduce an auxiliary topological game, and relate this game to the cops and robber game. More precisely, given a compact connected orientable surface $S$, the topological game yields a \emph{value} $v(S) \in \mathbb N \cup \{\infty\}$, and we show that if a graph $G$ can be drawn on $S$, then $c(G) \leq v(S) + 1$. Thus, to prove Theorem \ref{t:main}, it would be sufficient to bound from above the score of a compact connected orientable surface $S$ of genus $g$ in an appropriate manner. In fact, we are able to find the exact value of such a surface, up to a small additive constant.
\begin{theorem}\label{t:PC}
Let $S$ be a compact connected orientable surface of genus $g$. Then 
\[
 \frac{4g}{3} +2 \leq v(S) \leq  \frac{4g}{3}  + \frac{10}{3}.
\]
\end{theorem}
We note that, Theorem \ref{t:main} does not quite follow directly from Theorem \ref{t:PC}, since we can only show that $c(G) \leq v(S) +1$. However, with a more careful analysis of our strategy we can improve the constant term.

So far, Schr\"oder's conjecture is only known to be true for $g \leq 3$, see \cite{L19}. The only known exact values for $c(g)$ are $c(0) = 3$ due to Aigner and Fromme \cite{AF84}, and $c(1) = 3$ due to a recent result by Lehner \cite{L19}. The latter result shows that there are values of $g$ for which the bound in Schr\"oder's conjecture is not tight, and raises the question whether the bound is asymptotically optimal.

In fact, it is not clear whether there is even a linear lower bound. The only reference to non-trivial lower bounds for $c(g)$ in the literature that we could find come from the survey paper of Bonato and Mohar \cite{BM17} who give the following lower bound, which comes from a random construction of Mohar \cite{M08}.

\begin{theorem}[Mohar]
\[
c(g) \geq g^{\frac{1}{2} - o(1)}.
\]
\end{theorem}

Mohar also goes on to conjecture that this bound is essentially tight

\begin{conj}[Mohar \cite{M08}]
For every $\epsilon > 0$ there exists $g_0$ such that, for every $g\geq g_0$
\[
g^{\frac{1}{2} - \epsilon} < c(g) < g^{\frac{1}{2} + \epsilon}.
\]
\end{conj}

\section{Topological background}\label{s:top}
Throughout this section let $S$ be a compact connected orientable surface. By the classification theorem for compact surfaces, $S$ is homeomorphic to $T_n$ (the $n$-fold torus) with a finite number of holes (i.e.\ a finite number of non-touching open disks $D_1,\ldots,D_k \subset T_n$ removed). The genus of $S$ is defined as the genus of $T_n$ (which is $n$ by definition), and the boundary $\partial(S)$  of $S$ consists of the $k$ simple closed curves $\partial D_i$ for $i \in [k]$ if $k \neq 0$. 
%Josh : I changed this slightly, I think it makes sense if $k=0$ as well, but perhaps I should just make an explicit case distinction

Next, we fix some terminology for the operation of cutting and pasting surfaces along a particular class of arcs and closed curves. Both procedures are completely standard and well-known. Still, it seems desirable to isolate and state precisely which topological facts we will need for the remainder of the paper. 

Any non-trivial connected subset of $\partial S$ will be called a \emph{boundary arc} or a \emph{$\partial$-arc} of $S$. Clearly, any $\partial$-arc will either be homeomorphic to an arc or a simple closed curve.

Given two points $x,y \in S$, not necessarily distinct, an \emph{\arc{x}{y}} is a continuous function $f: [0,1] \rightarrow S$ such that 
\begin{itemize}
\item $f(0)=x$ and $f(1)=y$;
\item $f( (0,1) ) \cap \partial S = \emptyset$ 
\item $f$ is injective except for possibly $f(0)=f(1)$.
\end{itemize}
Given an \arc{x}{y} $f$, we will write $\partial f$ for $\{x,y\}$. When there is no room for confusion, it will often be convenient to talk about the image $A = f([0,1])$ of an \arc{x}{y} rather than the function $f$. In this case $\partial A$ is understood to be $\partial f$, which would otherwise be ambiguous in the case where $x=y$. See Figure~\ref{fig_typCutarc} for a typical such cutline.

\begin{figure}
  \begin{tikzpicture}[shading=ball, ball color=lightgray]
    % draw the sphere
    \draw[fill=lightgray] ellipse (5 and 2);
    
    \draw[fill=white] (-3,1) ellipse (.5 and .2);
     \draw[fill=white] (-2,-1) ellipse (.5 and .2);
      \draw[fill=white] (2,0) ellipse (.5 and .2);
%    \draw[fill=white] ellipse (5 and 2);
    
     \node (A) at (-2.5,1.4) {\rotatebox{10}{\textcolor{red}{\ding{33}}}};
 \node (B) at (-2,-.3) {\textcolor{red}{$A$}};
    \draw[dashed, red] (-2.8,1.2) .. controls (1,2.5) and (-2.2,-.5) .. (-2,-0.8);

     \begin{scope}[shift={(-.5,-1)}]
   \draw[dashed] (0,0)  coordinate(c1) {[yscale=.5] arc(180:0:.2)} coordinate(c2);
    \draw[dashed ](1,0)  coordinate(d1) {[yscale=.5] arc(180:0:.2)} coordinate(d2);

    \draw[fill opacity=.8,shade]
      (c1) to[out=70,in=110,looseness=2.2] (d2)
      {[yscale=.5] arc(360:180:.2)}
      to[out=110,in=70,looseness=2.5] (c2)
      {[yscale=.5] arc(360:180:.2)} -- cycle;
    \end{scope}

         \begin{scope}[shift={(1,1)}]
   \draw[dashed] (0,0)  coordinate(c1) {[yscale=.5] arc(180:0:.2)} coordinate(c2);
    \draw[dashed ](1,0)  coordinate(d1) {[yscale=.5] arc(180:0:.2)} coordinate(d2);

    \draw[fill opacity=.8,shade]
      (c1) to[out=70,in=110,looseness=2] (d2)
      {[yscale=.5] arc(360:180:.2)}
      to[out=110,in=70,looseness=2.3] (c2)
      {[yscale=.5] arc(360:180:.2)} -- cycle;
    \end{scope}

       \begin{scope}[shift={(-2,.5)}]
       
   \draw[dashed] (0,0)  coordinate(c1) {[yscale=.5] arc(180:0:.2)} coordinate(c2);
    \draw[dashed ](1,0)  coordinate(d1) {[yscale=.5] arc(180:0:.2)} coordinate(d2);

    \draw[fill opacity=.8,shade]
      (c1) to[out=70,in=110,looseness=2] (d2)
      {[yscale=.5] arc(360:180:.2)}
      to[out=110,in=70,looseness=2.3] (c2)
      {[yscale=.5] arc(360:180:.2)} -- cycle;

           \end{scope}
  \end{tikzpicture}
  \caption{An orientable surface $S$ of genus $3$ with $4$ boundary components and a dashed \arc{x}{y} $A$.}
  \label{fig_typCutarc}
\end{figure}
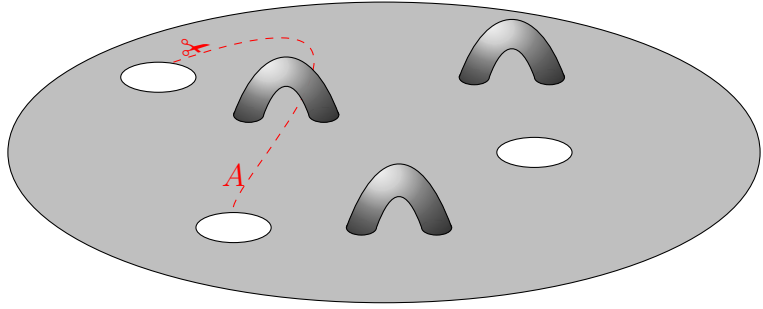

Now we introduce the space $\cut (S,f)$ obtained by cutting $S$ along $f$. Intuitively, cutting along $f$ will result in two new arcs in the boundary corresponding to the two sides of $f$. One way to formalise this intuition is to recall (see e.g.\ \cite[Theorem~3.1.1]{mohar2001graphs}) that any surface is homeomorphic to a triangulated surface, that is, a surface obtained by starting with a finite collection of triangles and identifying each edge of every triangle with at most one edge of another triangle (in particular the lengths of edges that get identified coincide). The boundary of a triangulated surface is the union of all edges that were not identified with another edge.

Without loss of generality, assume that $S$ is given as a triangulated surface, and that $A=f([0,1])$ is a polygonal arc or closed curve. By passing to a suitable refinement of the triangulation we may assume that this arc consists of edges of triangles which were identified.

\begin{defn}
\label{defn:pasting}
The space $\cut (S,f)$ is obtained from $S$ by un-identifying all edges that lie in $A$. We say $\cut (S,f)$ is obtained from $S$ by cutting along $A$.

The \emph{pasting map} is the map $\phi\colon \cut(S,f) \to S$ whose restriction to each individual triangle is the identity. This is a continuous map which is the identity on $S \setminus A$. 

The preimage of $A$ under $\phi$ consists of two arcs or closed curves which we will denote by $A_1$ and $A_2$, cf.\ Figure~\ref{fig_easycutting}. For later use, for $X \subseteq S$, let us define  $\cut(X,f) = \phi^{-1}(X)$.
\end{defn}

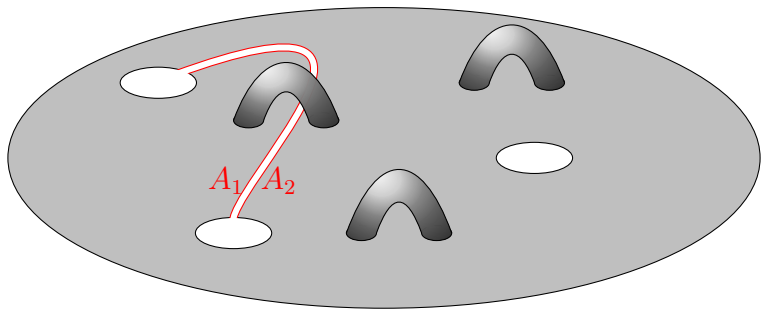
\begin{figure}
  \begin{tikzpicture}[shading=ball, ball color=lightgray]
    % draw the sphere
    \draw[fill=lightgray] ellipse (5 and 2);
    
    \draw[thick, fill=white] (-3,1) ellipse (.5 and .2);
     \draw[thick, fill=white] (-2,-1) ellipse (.5 and .2);
      \draw[thick, fill=white] (2,0) ellipse (.5 and .2);
%    \draw[fill=white] ellipse (5 and 2);
%    \node (A) at (-2.5,1.4) {\rotatebox{10}{\textcolor{red}{\ding{33}}}};
    
%    \draw[dashed, red] (-2.8,1.2) .. controls (1,2.5) and (-2.2,-.5) .. (-2,-0.8); 
     \draw[double distance=.8mm, red] (-2.8,1.1) .. controls (1,2.5) and (-2.2,-.5) .. (-2,-0.9);   
     
     \node (A1) at (-2.1,-.3) {\textcolor{red}{$A_1$}};
     
     \node (A2) at (-1.4,-.3) {\textcolor{red}{$A_2$}};
     
      \draw[draw=none, fill=white] (-3,1) ellipse (.5 and .2);
     \draw[draw=none, fill=white] (-2,-1) ellipse (.5 and .2);
      \draw[draw=none, fill=white] (2,0) ellipse (.5 and .2);

     \begin{scope}[shift={(-.5,-1)}]
   \draw[dashed] (0,0)  coordinate(c1) {[yscale=.5] arc(180:0:.2)} coordinate(c2);
    \draw[dashed ](1,0)  coordinate(d1) {[yscale=.5] arc(180:0:.2)} coordinate(d2);
    
   %    \useasboundingbox (-3,-3) rectangle (5,5); % To keep the image size reasonable
    \draw[fill opacity=.8,shade]
      (c1) to[out=70,in=110,looseness=2.2] (d2)
      {[yscale=.5] arc(360:180:.2)}
      to[out=110,in=70,looseness=2.5] (c2)
      {[yscale=.5] arc(360:180:.2)} -- cycle;
    \end{scope}

         \begin{scope}[shift={(1,1)}]
   \draw[dashed] (0,0)  coordinate(c1) {[yscale=.5] arc(180:0:.2)} coordinate(c2);
    \draw[dashed ](1,0)  coordinate(d1) {[yscale=.5] arc(180:0:.2)} coordinate(d2);
    
    %   \useasboundingbox (-3,-3) rectangle (5,5); % To keep the image size reasonable
    \draw[fill opacity=.8,shade]
      (c1) to[out=70,in=110,looseness=2] (d2)
      {[yscale=.5] arc(360:180:.2)}
      to[out=110,in=70,looseness=2.3] (c2)
      {[yscale=.5] arc(360:180:.2)} -- cycle;
    \end{scope}

       \begin{scope}[shift={(-2,.5)}]
       
   \draw[dashed] (0,0)  coordinate(c1) {[yscale=.5] arc(180:0:.2)} coordinate(c2);
    \draw[dashed ](1,0)  coordinate(d1) {[yscale=.5] arc(180:0:.2)} coordinate(d2);
    
     %  \useasboundingbox (-3,-3) rectangle (5,5); % To keep the image size reasonable
    \draw[fill opacity=.8,shade]
      (c1) to[out=70,in=110,looseness=2] (d2)
      {[yscale=.5] arc(360:180:.2)}
      to[out=110,in=70,looseness=2.3] (c2)
      {[yscale=.5] arc(360:180:.2)} -- cycle;

           \end{scope}

  \end{tikzpicture}
  \caption{The surface $\cut(S,f)$ with two sides $A_1$ and $A_2$, where $S$ and $A$ are as above has one fewer boundary component.}
  \label{fig_easycutting}
\end{figure}

\begin{lemma}
Let $S$ be a compact connected orientable surface, $x,y \in S$, and let $f$ be an \arc{x}{y} in $S$. Then $\cut(S,f)$ has at most two connected components, each of which is a compact orientable surface with boundary.
\end{lemma}

\begin{proof}
The space $\cut(S,f)$ is compact since it is the union of finitely many triangles and orientable since it is homeomorphic to a subspace of $S$.
%(Josh) : Is there a better way to see that it's orientable?
%(Josh) : If we want to be more precise, would have to define orientations of triangulations, and then the point is the the orientation of the triangulation $\cut(S,f)$ given locally by the orientation of $S$ witnesses that $\cut(S,f)$ is orientable.
Any connected component of $\cut(S,f)$ is hence a compact orientable triangulated surface with boundary.  It only remains to show that there are at most $2$ connected components.

Let $A_1$ and $A_2$ be as in Definition \ref{defn:pasting}. Since $A_1$ and $A_2$ are connected, each of them is contained in a connected component. Any connected component of $\cut (S,f)$ which contains neither of the two arcs would also be a non-trivial connected component of $S$, contradicting the fact that $S$ is connected.
\end{proof}

Note that, in particular, both $A_1$ and $A_2$ are $\partial$-arcs of the components of $\cut(S,f)$. The space $\cut(S,f)$ depends on whether the points $x$ and $y$ are distinct, whether they lie in $\partial S$ or not, and if they both do lie in $\partial S$, whether they lie in the same component of the boundary. We will be interested in how this operation changes the boundary and genus of the surface, and so in the next two lemmas will give a description of what happens in each of the various cases. Recall that, if we have a triangulation of a surface $S$ with $V$ vertices, $E$ edges and $F$ faces (i.e.\ triangles plus boundary cycles), then by Euler's formula
\begin{equation}\label{e:Euler}
g(S) = 1 - \frac{(V-E+F)}{2}.
\end{equation}

\begin{lemma}\label{l:x=yorC=C'}
Let $S$ be a compact orientable surface with (perhaps empty) boundary $\partial S = \{ \partial D_1, \partial D_2, \ldots, \partial D_m\}$. Let $x,y \in S$ be such that either $x=y$, or there exists an $i$ such that $x,y \in \partial D_i$. Let $f$ be an \arc{x}{y} in $S$ and $A=f([0,1])$, $A_1$ and $A_2$ be as above. If $x=y \not\in \partial S$ let $D = \{x\}$, otherwise let $D=D_i$.

Then the boundary of $\cut(S,f)$ is $\{ \partial D_j \colon D_j \neq D\} \cup \hat{D}_1 \cup \hat{D}_2$ where $\hat{D}_1$ consists of the arc $A_1$ together with one of the components of $\partial D \setminus \partial f$ and $\hat{D}_2$ consists of $A_2$ together with the other component, where if $D = \{x\}$ these components are both considered to be $\{x\}$.
%(Josh) : Here and in the next lemma do I actually need this final part, since if D= \{x\} then $A_1$ and $A_2$ already contain $x$...
%Perhaps I should consider them to be empty in this case, it doesn't really matter, I guess \partial D isn't defined and so needs `defining'

Furthermore, if $\cut(S,f)$ is connected, then its genus is one less than the genus of $S$, cf.\ Figure~\ref{fig_genuscutting}. If $\cut(S,f)$ is disconnected, then the sum of the genus of the components of $\cut(S,f)$ is equal to the genus of $S$, cf.\ Figure~\ref{fig_nogenuscutting}.
\end{lemma}

\begin{figure}
 \begin{tikzpicture}[shading=ball, ball color=lightgray]
    % draw the sphere
    \draw[fill=lightgray] ellipse (5 and 2);
    
    \draw[thick, fill=white] (-3,1) ellipse (.5 and .2);
     \draw[thick, fill=white] (-2,-1) ellipse (.5 and .2);
      \draw[thick, fill=white] (2,0) ellipse (.5 and .2);
%    \draw[fill=white] ellipse (5 and 2);
%    \node (A) at (-2.5,1.4) {\rotatebox{10}{\textcolor{red}{\ding{33}}}};
    
%    \draw[dashed, red] (-2.8,1.2) .. controls (1,2.5) and (-2.2,-.5) .. (-2,-0.8); 
     \draw[double distance=1mm, red] (-2.8,1.1) .. controls (0.5,3) and (-2.2,-2.5) .. (-3,1);   
     
     \node (A1) at (-2.3,.2) {\textcolor{red}{$A_1$}};
     
     \node (A2) at (-1.6,-.2) {\textcolor{red}{$A_2$}};
     
      \draw[draw=none, fill=white] (-3,1) ellipse (.5 and .2);
     \draw[draw=none, fill=white] (-2,-1) ellipse (.5 and .2);
      \draw[draw=none, fill=white] (2,0) ellipse (.5 and .2);

     \begin{scope}[shift={(-.5,-1)}]
   \draw[dashed] (0,0)  coordinate(c1) {[yscale=.5] arc(180:0:.2)} coordinate(c2);
    \draw[dashed ](1,0)  coordinate(d1) {[yscale=.5] arc(180:0:.2)} coordinate(d2);
    
   %    \useasboundingbox (-3,-3) rectangle (5,5); % To keep the image size reasonable
    \draw[fill opacity=.8,shade]
      (c1) to[out=70,in=110,looseness=2.2] (d2)
      {[yscale=.5] arc(360:180:.2)}
      to[out=110,in=70,looseness=2.5] (c2)
      {[yscale=.5] arc(360:180:.2)} -- cycle;
    \end{scope}

         \begin{scope}[shift={(1,1)}]
   \draw[dashed] (0,0)  coordinate(c1) {[yscale=.5] arc(180:0:.2)} coordinate(c2);
    \draw[dashed ](1,0)  coordinate(d1) {[yscale=.5] arc(180:0:.2)} coordinate(d2);
    
    %   \useasboundingbox (-3,-3) rectangle (5,5); % To keep the image size reasonable
    \draw[fill opacity=.8,shade]
      (c1) to[out=70,in=110,looseness=2] (d2)
      {[yscale=.5] arc(360:180:.2)}
      to[out=110,in=70,looseness=2.3] (c2)
      {[yscale=.5] arc(360:180:.2)} -- cycle;
    \end{scope}

       \begin{scope}[shift={(-2,.5)}]
       
   \draw[dashed] (0,0)  coordinate(c1) {[yscale=.5] arc(180:0:.2)} coordinate(c2);
    \draw[dashed ](1,0)  coordinate(d1) {[yscale=.5] arc(180:0:.2)} coordinate(d2);
    
     %  \useasboundingbox (-3,-3) rectangle (5,5); % To keep the image size reasonable
    \draw[fill opacity=.8,shade]
      (c1) to[out=70,in=110,looseness=2] (d2)
      {[yscale=.5] arc(360:180:.2)}
      to[out=110,in=70,looseness=2.3] (c2)
      {[yscale=.5] arc(360:180:.2)} -- cycle;

           \end{scope}

        \node () at (0,-2.5) {$\cong$};

        \begin{scope}[shift={(0,-5)}]
           % draw the sphere
    \draw[fill=lightgray] ellipse (5 and 2);
    
    \draw[thick, fill=white] (-3,1) ellipse (.5 and .2);
     \draw[thick, fill=white] (-2,-1) ellipse (.5 and .2);
      \draw[thick, fill=white] (2,0) ellipse (.5 and .2);
%    \draw[fill=white] ellipse (5 and 2);
%    \node (A) at (-2.5,1.4) {\rotatebox{10}{\textcolor{red}{\ding{33}}}};
    
%    \draw[dashed, red] (-2.8,1.2) .. controls (1,2.5) and (-2.2,-.5) .. (-2,-0.8); 
     \draw[double distance=1mm, red] (-2.8,1.1) .. controls (0.5,3) and (-2.2,-2.5) .. (-3,1);

      \draw[white, fill=white] (-2.8,1.1) .. controls (0.5,3) and (-2.2,-2.5) .. (-3,1) -- cycle;   
     
  %   \node (A1) at (-2.3,.2) {\textcolor{red}{$A_1$}};
     
     \node (A2) at (-1.6,-.2) {\textcolor{red}{$A_2$}};
     
      \draw[draw=none, fill=white] (-3,1) ellipse (.5 and .2);
     \draw[draw=none, fill=white] (-2,-1) ellipse (.5 and .2);
      \draw[draw=none, fill=white] (2,0) ellipse (.5 and .2);

     \begin{scope}[shift={(-.5,-1)}]
   \draw[dashed] (0,0)  coordinate(c1) {[yscale=.5] arc(180:0:.2)} coordinate(c2);
    \draw[dashed ](1,0)  coordinate(d1) {[yscale=.5] arc(180:0:.2)} coordinate(d2);
    
   %    \useasboundingbox (-3,-3) rectangle (5,5); % To keep the image size reasonable
    \draw[fill opacity=.8,shade]
      (c1) to[out=70,in=110,looseness=2.2] (d2)
      {[yscale=.5] arc(360:180:.2)}
      to[out=110,in=70,looseness=2.5] (c2)
      {[yscale=.5] arc(360:180:.2)} -- cycle;
    \end{scope}

         \begin{scope}[shift={(1,1)}]
   \draw[dashed] (0,0)  coordinate(c1) {[yscale=.5] arc(180:0:.2)} coordinate(c2);
    \draw[dashed ](1,0)  coordinate(d1) {[yscale=.5] arc(180:0:.2)} coordinate(d2);
    
    %   \useasboundingbox (-3,-3) rectangle (5,5); % To keep the image size reasonable
    \draw[fill opacity=.8,shade]
      (c1) to[out=70,in=110,looseness=2] (d2)
      {[yscale=.5] arc(360:180:.2)}
      to[out=110,in=70,looseness=2.3] (c2)
      {[yscale=.5] arc(360:180:.2)} -- cycle;
    \end{scope}

      \begin{scope}[shift={(-2,.5)}]
%       
%   \draw[dashed] (0,0)  coordinate(c1) {[yscale=.5] arc(180:0:.2)} coordinate(c2);
        \node (A1) at (1.5,.4) {\textcolor{red}{$A_1$}};
 \draw[red, fill=white] (1.8,0)  coordinate(c1) {[yscale=.5] arc(0:360:.4)} coordinate(c2);
  \draw (1.8,0)  coordinate(c1) {[yscale=.5] arc(360:180:.4)} coordinate(c2);
%    \draw[dashed ](1,0)  coordinate(d1) {[yscale=.5] arc(180:0:.2)} coordinate(d2);
%    
%     %  \useasboundingbox (-3,-3) rectangle (5,5); % To keep the image size reasonable
%    \draw[fill opacity=.8,shade]
%      (c1) to[out=70,in=110,looseness=2] (d2)
%      {[yscale=.5] arc(360:180:.2)}
%      to[out=110,in=70,looseness=2.3] (c2)
%      {[yscale=.5] arc(360:180:.2)} -- cycle;
%
           \end{scope}

        \end{scope}

  \end{tikzpicture}

  \caption{When the \arc{x}{y} $f$ attaches to the same boundary component and $\cut(S,f)$ is connected, the cut is genus reducing.}
  \label{fig_genuscutting}
\end{figure}
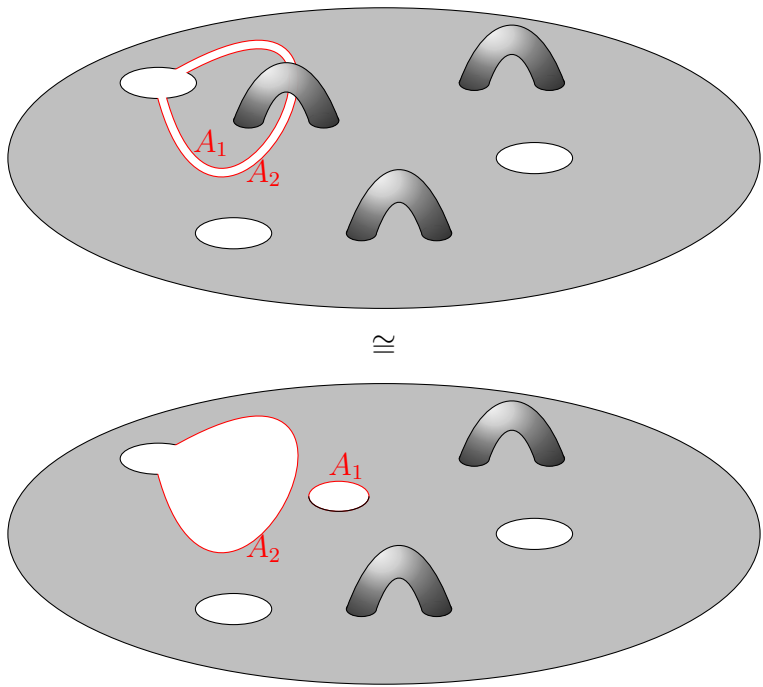

\begin{figure}
   \begin{tikzpicture}[shading=ball, ball color=lightgray]
    % draw the sphere
    \draw[fill=lightgray] ellipse (5 and 2);
    
    \draw[thick, fill=white] (-3,1) ellipse (.5 and .2);
     \draw[thick, fill=white] (-2,-1) ellipse (.5 and .2);
      \draw[thick, fill=white] (2,0) ellipse (.5 and .2);
%    \draw[fill=white] ellipse (5 and 2);
%    \node (A) at (-2.5,1.4) {\rotatebox{10}{\textcolor{red}{\ding{33}}}};
    
%    \draw[dashed, red] (-2.8,1.2) .. controls (1,2.5) and (-2.2,-.5) .. (-2,-0.8); 
     \draw[double distance=1mm, red] (-2.8,1.1) .. controls (2,3) and (-0.7,-2.5) .. (-3,1);   
     
     \node (A1) at (-1.4,.1) {\textcolor{red}{$A_1$}};
     
     \node (A2) at (-2,-.3) {\textcolor{red}{$A_2$}};
     
      \draw[draw=none, fill=white] (-3,1) ellipse (.5 and .2);
     \draw[draw=none, fill=white] (-2,-1) ellipse (.5 and .2);
      \draw[draw=none, fill=white] (2,0) ellipse (.5 and .2);

     \begin{scope}[shift={(-.5,-1)}]
   \draw[dashed] (0,0)  coordinate(c1) {[yscale=.5] arc(180:0:.2)} coordinate(c2);
    \draw[dashed ](1,0)  coordinate(d1) {[yscale=.5] arc(180:0:.2)} coordinate(d2);
    
   %    \useasboundingbox (-3,-3) rectangle (5,5); % To keep the image size reasonable
    \draw[fill opacity=.8,shade]
      (c1) to[out=70,in=110,looseness=2.2] (d2)
      {[yscale=.5] arc(360:180:.2)}
      to[out=110,in=70,looseness=2.5] (c2)
      {[yscale=.5] arc(360:180:.2)} -- cycle;
    \end{scope}

         \begin{scope}[shift={(1,1)}]
   \draw[dashed] (0,0)  coordinate(c1) {[yscale=.5] arc(180:0:.2)} coordinate(c2);
    \draw[dashed ](1,0)  coordinate(d1) {[yscale=.5] arc(180:0:.2)} coordinate(d2);
    
    %   \useasboundingbox (-3,-3) rectangle (5,5); % To keep the image size reasonable
    \draw[fill opacity=.8,shade]
      (c1) to[out=70,in=110,looseness=2] (d2)
      {[yscale=.5] arc(360:180:.2)}
      to[out=110,in=70,looseness=2.3] (c2)
      {[yscale=.5] arc(360:180:.2)} -- cycle;
    \end{scope}

       \begin{scope}[shift={(-2,.5)}]
       
   \draw[dashed] (0,0)  coordinate(c1) {[yscale=.5] arc(180:0:.2)} coordinate(c2);
    \draw[dashed ](1,0)  coordinate(d1) {[yscale=.5] arc(180:0:.2)} coordinate(d2);
    
     %  \useasboundingbox (-3,-3) rectangle (5,5); % To keep the image size reasonable
    \draw[fill opacity=.8,shade]
      (c1) to[out=70,in=110,looseness=2] (d2)
      {[yscale=.5] arc(360:180:.2)}
      to[out=110,in=70,looseness=2.3] (c2)
      {[yscale=.5] arc(360:180:.2)} -- cycle;

           \end{scope}

  \end{tikzpicture}

  \caption{When the \arc{x}{y} $f$ attaches to the same boundary component and $\cut(S,f)$ is disconnected.}
  \label{fig_nogenuscutting}
\end{figure}
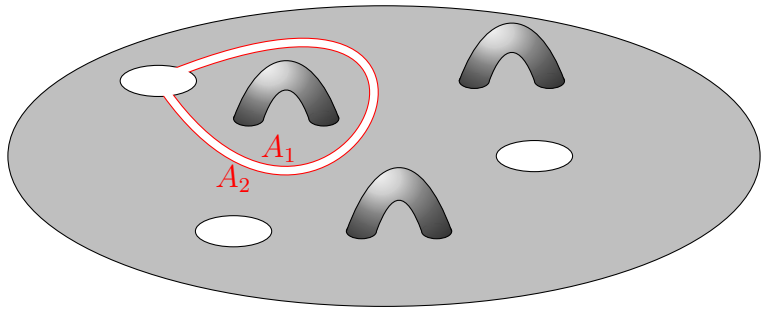

\begin{proof}
The statement about the structure of the boundary is true by construction, using the observation that the boundary of a triangulated surface consists of those edges that are not identified with any other edge. 

Suppose that $A$ consists of $k$ edges. If $\cut(S,f)$ is connected and $x,y \in \partial S$, then the bound for the genus follows from (\ref{e:Euler}) by noting that compared to $S$ the number of faces has increased by one, the number of edges has increased by $k$, and the number of vertices has increased by $k+1$. If $x \in \Int{S}$ (and so $x=y$) then the number of faces has increased by two, the number of edges has increased by $k$ and the number of vertices has increased by $k$, and the results again follows from (\ref{e:Euler}).

Finally if $\cut(S,f)$ is disconnected, with components $S'_1$ and $S'_2$ then a similar argument shows that $g(S) = g(S'_1) + g(S'_2)$. \qedhere
%Explicitly: If $G$ is the triangulation of $S$ and $G'_1$ and $G'_2$ are the induced triangulations of $S'_1$ and $S'_2$ then we have that $f(G) = f(G'_1) + f(G'_2) -1$, $e(G) = e(G'_1) + e(G'_2) - k$ and $v(G) = v(G'_1) + v(G'_2) - (k+1)$ if $x \not\in \Int{S}$ and $f(G) = f(G'_1) + f(G'_2) - 2$, $e(G) = e(G'_1) + e(G'_2) - k$ and $v(G) = v(G'_1) + v(G'_2) - k$, and again the result follows from (\ref{e:Euler}).
\end{proof}

\begin{lemma}\label{l:xneqyandCneqC'}
Let $S$ be a compact orientable surface with (perhaps empty) boundary $\partial S = \{ \partial D_1, \partial D_2, \ldots, \partial D_m\}$. Let $x, y \in S$ be such that there is no $k$ with $\{x,y\} \subseteq D_k$. Let $f$ be an \arc{x}{y} in $S$ and $A=f([0,1]), A_1$ and $A_2$ be as above. If $x \not\in \partial S$ let $D = \{x\}$, otherwise let $D=D_i$ such that $x \in D_i$. Similarly, if $y \not\in \partial S$ let $D' = \{y\}$, otherwise let $D'=D_j$ such that $y \in D_j$.

Then the boundary of $\cut(S,f)$ is $\{ \partial D_k \colon  D_k \neq D,D'\} \cup \hat{D}$ where $\hat{D}$ consists of the arcs $A_1$ and $A_2$ together with $\partial D \setminus \partial f$ and $\partial D' \setminus \partial f$ where if $D= \{x\}$ or $D'=\{y\}$ we let $\partial D$ and $\partial D'$ be $\{x\}$ and $\{y\}$ respectively.
%(Josh) : same issue as the lemma above.

Furthermore, $\cut (S,f)$ is connected and the genus of $\cut (S,f)$ equals the genus of $S$, cf.\ Figure~\ref{fig_easycutting}.
\end{lemma}
\begin{proof}
As in the proof of the previous lemma, the statement about structure of the boundary is true by construction, the statement about the genus follows from (\ref{e:Euler}).
\end{proof}
 
We shall also need a slight technical extension of Lemma \ref{l:x=yorC=C'} which says that in the case where there exists an $i$ such that $x,y \in \partial D_i$ then for any component $Z$ of $\partial D_i \setminus \{x,y\}$ we can choose an \arc{x}{y} which lies `along $Z$', so that the component of $\cut(S,f)$ containing $Z$ is planar.

\begin{lemma}\label{l:SameCyclesSafe}
Let $S$ be a compact orientable surface with boundary $\partial S = \{ \partial D_1, \partial D_2, \ldots, \partial D_m\}$. Let $x,y \in S$ be such that there exists an $i$ such that $x,y \in \partial D_i$ and let $Z_1$ and $Z_2$ be the components of $ \partial D_i \setminus \partial f$.

Then there exists an \arc{x}{y} in $S$, with $A=f([0,1]), A_1$ and $A_2$ as above such that $\cut(S,f)$ has two components
\begin{itemize}
    \item One of genus $0$ whose boundary consists of a cycle formed by $A_1$ and $Z_1$;
    \item One of genus $g(S)$ whose boundary consist of the $\partial D_j$ for $j \neq i$ together with a cycle formed by $A_2$ and $Z_2$.
\end{itemize}
\end{lemma}

\begin{proof}
Let $K$ be an annulus, that is, a closed disk $E$ with an open disk $F$ removed from its interior. Let $\psi$ be a homeomorphism from $E$ to $D_i$, and let $S'$ be the space obtained by quotienting $S \dot\cup K$ by $\psi$. 

It is clear that $S'$ is a compact orientable surface with the same number of holes as $S$, and so by the classification theorem for compact surfaces there is a homeomorphism $\Psi$ from $S$ to $S'$. Let $\hat{x} = \Psi(x)$ and $\hat{y} = \Psi(y)$ and note that $\hat{x},\hat{y} \in \partial E$.

There is some \arc{\hat{x}}{\hat{y}} $\hat{f}$ in $K$ and by Lemma \ref{l:x=yorC=C'} there are two components of $\cut(K,\hat{f})$, both of which are planar and one of which doesn't meet $\partial E$. We may assume that this component is the one which contains $\Psi(Z_1)$, since there is a homeomorphism of $K$ exchanging $\Psi(Z_1)$ and $\Psi(Z_2)$.

If we let $f = \Psi^{-1} \circ \hat{f}$ we see that $f$ is indeed an \arc{x}{y} satisfying the first bullet point, since the component of $\cut(S,f)$ containing $Z_1$ is just the image of the component of $\cut(K,\hat{f})$ containing $\Psi(Z_1)$. It then follows from Lemma \ref{l:x=yorC=C'} that the second bullet point is also true.
\end{proof}

\section{Relating cops and robbers to a topological game}

\subsection{The topological Marker-Cutter game} As announced in the introduction, we will now introduce a two-player game player on a compact orientable surface $S$, and then relate strategies in this topological game on $S$ to cop strategies in the cop and robber game for graphs embeddable on $S$. 

\begin{defn}[Topological Marker-Cutter game]\label{d:topgame}
The \emph{Topological Marker-Cutter game} is a game played by two players, Marker and Cutter, on some compact orientable surface $S$. We let $S_0 = S$ and $X_0 =\emptyset$.

In a general turn we have some compact orientable surface with boundary $S_k$ together with some finite set $X_k \subset \partial S_k$. First, Marker chooses two, not necessarily distinct, points $x$ and $y$ in $X_k \cup \Int{S_k}$, where $\Int{S_k}$ is the interior of $S_k$. Cutter responds by choosing an \arc{x}{y} $f_{k+1}$ in $S_k$. Cutter also chooses a component $S_{k+1}$ of $\cut(S_k,f_{k+1})$ and we set $X_{k+1} := \phi^{-1}(X_k \cup \{x,y\})\cap S_{k+1}$, where $\phi \colon \cut(S_k,f_{k+1}) \to S_k$ is the pasting map. Let $A_{k+1}, A'_{k+1} \subset \cut(S_k,f_{k+1})$ be the boundary arcs originating from $f_{k+1}([0,1])$ in $S_k$.

In this way a \emph{play} of the game can be represented by a pair of sequences $(f_i, S_i \colon i \in \mathbb{N})$ where each $f_i$ is a \cutline\ in $S_{i-1}$.  Given a play of the game $(f_i, S_i \colon i \in \mathbb{N})$ the set of \emph{active boundary arcs} at turn $k$ is
\[
\mathcal{A}_k = \{ A_i \colon A_i \subset \partial S_k \} \cup \{ A'_i \colon A'_i \subset \partial S_k \},
\]
the set of \emph{active indices} is $I_k = \{ i \colon A_i \in \mathcal{A}_k \text{   or   } A'_i \in \mathcal{A}_k \}$ and we let $a_k = |I_k|$. Given a play the \emph{score} of the game is $\sup_{i \in \mathbb{N}} a_i$, the supremum of the number of active indices at any turn in the game. Marker's aim in the game is to minimise the score and Cutter's aim is to maximise the score.

For a given surface $S$ let us define
\[
v(S) = \min \{t \colon \text{ Marker has a strategy to limit the score to at most } t \}.
\]
If Marker has no strategy bounding the score of the game, then we let $v(S)$ be infinite.

\end{defn}

The following lemma follows straightforwardly from Lemmas \ref{l:x=yorC=C'} and \ref{l:xneqyandCneqC'}.

\begin{lemma}\label{l:boundaryprop}
For a partial play $(f_i, S_i \colon i \in [k])$ in the topological Marker-Cutter game we have the following facts:
\begin{itemize}
\item $\partial S_k \subseteq \bigcup_{i\in [k]} (A_i \cup A'_i) \cup \partial S$;
\item Every component of $\partial S_k$ is homeomorphic to a cycle and each active boundary arc $B \in \mathcal{A}_k$ is a segment of one of these cycles;
\item $X_k = \bigcup_{B \in \mathcal{A}_k} \partial B$.
\end{itemize}
\end{lemma}

\subsection{Bounding the cop number in terms of the Marker-Cutter game}
We now show how to transfer a strategy for Marker in the topological Marker-Cutter game on a surface $S$ to a strategy for the cops in the cops and robbers game on a graph drawn on $S$. In this way we will show that Theorem \ref{t:main} follows from Theorem \ref{t:PC}.

We say that a cop \emph{guards} a set $C$ of vertices of $G$ if whenever the robber moves to a vertex of $C$ he is caught by that cop on the next move. For a subgraph $H$ of $G$, we say that a cop guards $C$ in $H$, if she guards $C$ given that the robber only moves along edges of $H$. We call $C \subseteq V(G)$ \emph{guardable}, if there is a strategy for a single cop $c$ in which, after finitely many steps, $c$ guards $C$, and for $R \subseteq V(G)$ and $H \subseteq G$, we call $C$ \emph{guardable relative to $R$ in $H$}, if after finitely many steps, during which the robber stays in $R$, the cop guards $C$ in $H$.

The following lemma, which originates from Aigner and Fromme \cite{AF84}, forms the cornerstone of essentially all known upper bounds on the cop number of graphs.

\begin{lemma}[\cite{AF84}, Lemma 4]
\label{lem:guardshortestpath}
For $x,y \in V(G)$, the vertex set of any shortest path from $x$ to $y$ is guardable.
\end{lemma}

%(Josh) : Perhaps mention this here?
This will also be the only tool we need to give our strategy for the cops. Furthermore, we remark that the near-optimality of Theorem \ref{t:PC} seems to suggest that Theorem \ref{t:main} is perhaps a natural limit, at least asymptotically, to what can be proved using only this tool.

\begin{lemma}
\label{lem:relativeguard}
Let $f\colon G' \to G$ be a homomorphism of graphs $G'$ and $G$, and let $C,R \subset V(G)$ with $C\cup R = V(G)$. Assume that every edge incident to $R$ is the image of some edge under $f$, and that the restriction of $f$ to $f^{-1}(R)$ is an isomorphism onto $R$. If $f^{-1}(C)$ is guardable relative to $f^{-1}(R)$ in $G'$, then $C$ is guardable relative to $R$ in $G$.
\end{lemma}

\begin{proof}
Let $C'= f^{-1}(C)$ and $R' = f^{-1}(R)$. Let us denote by $r \in R$ the robber's initial position, and let $r' = f^{-1}(r)$. We may assume that the cop is at some position $c$ such that $c' = f^{-1}(c)$ exists. 

Let us imagine that we have a `fake-robber' at $r'$ and a `fake-cop' at $c'$ in $G'$. Our strategy for the cop will be to maintain a game state in the fake game such that $r' = f^{-1}(r)$ and $c' = f^{-1}(c)$. By assumption the fake-cop has a strategy which ensures that she guards $C'$ in $G'$ after finitely many steps, assuming the fake-robber stays in $R'$. 

Each time the robber moves in $R$, since the restriction of $f$ to $R'$ is an isomorphism onto $R$, there is a legal move for the fake-robber which maintains the property that the fake-robber is at the image of the robber's location under $f^{-1}$. Note that, whilst the robber stays in $R$, then under this strategy the fake-robber stays in $R'$.

As a response to this move of fake-robber, fake-cop's strategy leads her to move somewhere in $G'$ and, since $f$ is a homomorphism, there is a legal move for the cop in $G$ which maintains the property that the cop is at the image of the fake-cop's location under $f$. Hence, as long as the robber stays in $R$, the cop can follow this strategy.

After a finite number of steps following this strategy the fake-cop guards $C'$ in $G'$. We claim that at this point, by continuing to mirror the strategy of the fake cop, the cop guards $C$ in $G$.

Indeed, suppose that after this point the robber moves to a vertex of $C$. Since $V(G) = R \cup C$, the first time this happens he moves along some edge $e=rx$, where $r \in R$ and $x \in C$, and so up to this point the robber was in $R$ and so the cop could continue to follow her strategy. By assumption, $e$ is the image of some edge under $f$, and so there is $x' \in C'$ with $f(x') = x$ and an edge $e'$ connecting $r'=f^{-1}(r)$ to $x'$. 

Hence, the fake-robber at $r'$ has a legal move to $x' \in C'$ in $G'$. However, since fake-cop is guarding $C'$ relative to $R'$, it follows that the current position $c'$ of the fake-cop in $G'$ is adjacent to $x'$. Our strategy guarantees that the cop is at the image of $c'$ under $f$, and since $f$ is a homomorphism it follows that $f(c')$ is adjacent to $f(x')=x$. Hence, the cop in $G$ can catch the robber on the next move, and thus is guarding $C$ in $G$.
\end{proof}

In the following definitions, by an \emph{arc} we always mean a non-trivial arc, i.e.\ a subspace homeomorphic to the closed unit interval.

\begin{defn}
Given a graph $G=(V,E)$ and a compact orientable surface $S$, an \emph{embedding of $G$ on $S$} is a map $\sigma$ such that:
\begin{itemize}
    \item $\sigma$ maps distinct vertices of $V$ to distinct points of $S$;
    \item $\sigma$ maps an edge $xy \in E$ to a simple $\sigma(x)-\sigma(y)$ arc in $S$;
    \item No inner point of $\sigma(xy)$ is the image of a vertex or lies on the image of another edge.
\end{itemize}
\end{defn}

The \emph{genus} of a graph $G$ can then be defined as the smallest $g$ such that $G$ has an embedding on a compact orientable surface $S$ with $g(S) =g$. It will be more convenient for us to work with a slightly different notion of embedding graphs on a surface.

\begin{defn}
\label{def:painting}
Given a graph $G$ and a compact orientable surface $S$, a \emph{painting of $G$ on $S$} is a family $\mathcal{P} = (D_v \colon v \in V(G))$ such that:
\begin{itemize}
    \item $D_v \subset S$ is homeomorphic to a closed disk for each $v \in V(G)$;
    \item If $uv \in E(G)$, then $D_u \cap D_v$ is a disjoint union of arcs contained in $\partial D_u \cap \partial D_v$;
    \item If $uv \notin E(G)$, then $D_u \cap D_v = \emptyset$;
    \item $D_u \cap D_v \cap D_w = \emptyset$ for any distinct  $u,v,w \in V(G)$.
\end{itemize}
\end{defn}

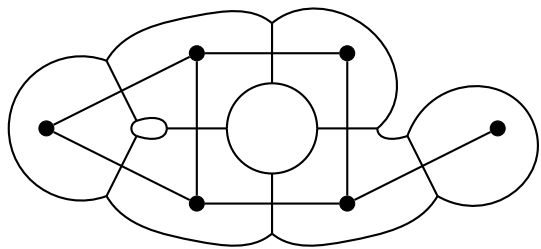
\begin{figure}[ht!]

\begin{tikzpicture}[use Hobby shortcut]

\node[draw,circle,inner sep=2,fill=black] (v1) at (-3,0){};
\node[draw,circle,inner sep=2,fill=black] (v2) at (-1,-1){};
\node[draw,circle,inner sep=2,fill=black] (v3) at (-1,1){};
\node[draw,circle,inner sep=2,fill=black] (v4) at (1,-1){};
\node[draw,circle,inner sep=2,fill=black] (v5) at (1,1){};
\node[draw,circle,inner sep=2,fill=black] (v6) at (3,0){};

%nodes for painting
\node[circle,inner sep=0] (v1out) at (-3.5,0){};
\node[circle,inner sep=0] (v2out) at (-1,-1.5){};
\node[circle,inner sep=0] (v3out) at (-1,1.5){};
\node[circle,inner sep=0] (v4out) at (1,-1.5){};
\node[circle,inner sep=0] (v5out) at (1,1.5){};
\node[circle,inner sep=0] (v6out) at (3.5,0){};

\node[circle,inner sep=0] (e12b) at (-2.2,-0.9){};
\node[circle,inner sep=0] (e12t) at (-1.8,-0.1){};
\node[circle,inner sep=0] (e13t) at (-2.2,0.9){};
\node[circle,inner sep=0] (e13b) at (-1.8,0.1){};
\node[circle,inner sep=0] (e23l) at (-1.4,0){};
\node[circle,inner sep=0] (e23r) at (-.6,0){};
\node[circle,inner sep=0] (e35t) at (0,1.4){};
\node[circle,inner sep=0] (e35b) at (0,.6){};
\node[circle,inner sep=0] (e24t) at (0,-.6){};
\node[circle,inner sep=0] (e24b) at (0,-1.4){};
\node[circle,inner sep=0] (e45l) at (.6,0){};
\node[circle,inner sep=0] (e45r) at (1.4,0){};
\node[circle,inner sep=0] (e46b) at (2.2,-0.9){};
\node[circle,inner sep=0] (e46t) at (1.8,-0.1){};

%intersections in painting
\path[draw,thick] (e12b.center)--(e12t.center);
\path[draw,thick] (e13b.center)--(e13t.center);
\path[draw,thick] (e23l.center)--(e23r.center);
\path[draw,thick] (e45l.center)--(e45r.center);
\path[draw,thick] (e35b.center)--(e35t.center);
\path[draw,thick] (e24b.center)--(e24t.center);
\path[draw,thick] (e46b.center)--(e46t.center);

%painting
\path[thick, draw,out angle=200, in angle=160] (e12b.center)..(v1out.center)..(e13t.center);
\path[thick, draw,out angle=160, in angle=200] (e12t.center)..(e13b.center);

\path[thick, draw,out angle=-60, in angle=220] (e12b.center)..(v2out.center)..(e24b.center);
\path[thick, draw,out angle=180, in angle=-90] (e24t.center)..(e23r.center);
\path[thick, draw,out angle=-90, in angle=-20] (e23l.center)..(e12t.center);

\path[name path=p1,thick, draw,out angle=60, in angle=-220] (e13t.center)..(v3out.center)..(e35t.center);
\path[name path=p2,thick, draw,out angle=180, in angle=90] (e35b.center)..(e23r.center);
\path[name path=p3, thick, draw,out angle=90, in angle=20] (e23l.center)..(e13b.center);

\path[thick, draw,out angle=-120, in angle=-40] (e46b.center)..(v4out.center)..(e24b.center);
\path[thick, draw,out angle=0, in angle=-90] (e24t.center)..(e45l.center);
\path[thick, draw,out angle=-90, in angle=200] (e45r.center)..(e46t.center);

\path[thick, draw,out angle=40, in angle=40] (e35t.center)..(v5out.center)..(e45r.center);
\path[thick, draw,out angle=90, in angle=0] (e45l.center)..(e35b.center);

\path[thick, draw,out angle=70, in angle=-30] (e46t.center)..(v6out.center)..(e46b.center);

%graph edges
\path[draw,thick] (v1)--(v2);
\path[draw,thick] (v1)--(v3);
\path[draw,thick] (v2)--(v3);
\path[draw,thick] (v2)--(v4);
\path[draw,thick] (v3)--(v5);
\path[draw,thick] (v4)--(v5);
\path[draw,thick] (v4)--(v6);

\end{tikzpicture}
\caption{A graph $G$ and a painting $\mathcal{P}$ of $G$}
\end{figure}

It is easy to see that a graph has an embedding on a compact surface of genus $g$ if and only if it has a painting on that surface. Furthermore, a set of closed disks in a surface is a painting of some graph, if and only if any triple has empty intersection, and the intersection of any pair is either empty or a disjoint union of arcs contained in the boundary of both. 
Indeed, given a set $\mathcal P$ with these properties, we can define the \emph{canonical graph $G_{\mathcal P}$} of $\mathcal P$ to be the graph with vertex set $\mathcal P$ and an edge between $D$ and $D'$ if $D \cap D' \neq \emptyset$. Note that any graph that has a painting $\mathcal P$ is isomorphic to $G_{\mathcal P}$.
%, so we don't lose any generality by only considering the canonical graph of a painting.
%

Given a painting $\mathcal P$ on a surface $S$, let $F_{\mathcal{P}}= \bigcup_{D \in \mathcal P} D$. Note that, viewed as a subspace of $S$, $F_{\mathcal P}$ is itself an orientable surface with boundary, and its boundary is a subset of the union of the boundaries of the disks in $\mathcal P$. %
%Further note that we can define the surface $F_{\mathcal P}$ without referring to $S$ from a set of disks satisfying Definition \ref{def:painting} by identifying appropriate arcs.
%In particular, when we say that $\mathcal P$ is a painting without mentioning the surface $S$ we mean that $\mathcal P$ is a set of disks satisfying the properties from Definition \ref{def:painting} interpreted as a painting on $F_{\mathcal P}$ as above.
%(Josh) : Above is not necessary I don't think anymore, and also a bit ambiguous since the surface $S$ is necessary to make sure $F_\mathcal{P}$ is orientable.
A \emph{painting homomorphism} between two paintings $\mathcal P$ and $\mathcal P'$ is a map $\psi\colon \mathcal P \to \mathcal P'$, such that for $D_1,D_2 \in \mathcal P$ which touch, $\psi(D_1)$ and $\psi(D_2)$ also touch. We extend, in the usual manner, the notion of homomorphism to define isomorphisms between paintings. It is worth noting that our notion of painting homomorphism is combinatorial rather than topological. In particular, any painting homomorphism gives a homomorphism of the canonical graphs, but it is possible that two paintings $\mathcal P$ and $\mathcal P'$ are isomorphic while $F_{\mathcal P}$ and $F_{\mathcal P'}$ are not homeomorphic. For instance, let $\mathcal P$ consist of two disks intersecting in a single arc and let $\mathcal P'$ consist of two disks intersecting in multiple arcs.

We say that a \cutline \ $f$ (cf.\ Section~\ref{s:top}) in $F_{\mathcal P}$ induces a walk $W=(D_1,D_2,\dots,D_n)$ in $G_{\mathcal P}$ if there is a homeomorphism $s \colon [0,1] \to [0,1]$ such that for $a = f \circ s$ and every $D \in \mathcal P$ we have \[
a^{-1}(D) = \bigcup _{\{i\colon D_i=D\}} \left[\frac {i-1}n, \frac {i}n\right],
\]
Note that in particular $f^{-1}(D) = \emptyset$ for all $D \not\in W$ since the corresponding union is empty. Furthermore, the only values such that $a(z)$ lies in $D \cap D'$ for $D \neq D'$ are $z= \frac in$ for $1 \leq i \leq n-1$. In particular, $f(0)$ and $f(1)$ lie in a unique disk of $\mathcal P$. Since we are assuming that $f$ is a \cutline \ in $F_{\mathcal P}$, the definition of \cutline \ further implies that $a(\frac in)$ is an interior point of one of the intervals in $D_i \cap D_{i+1}$.

\begin{figure}[ht!]
\begin{tikzpicture}[use Hobby shortcut]

\begin{scope}[shift={(-4,0)}]

\node[draw,circle,inner sep=2,fill=black] (v1) at (-3,0){};
\node[draw,circle,inner sep=2,fill=black] (v2) at (-1,-1){};
\node[draw,circle,inner sep=2,fill=black] (v3) at (-1,1){};
\node[draw,circle,inner sep=2,fill=black] (v4) at (1,-1){};
\node[draw,circle,inner sep=2,fill=black] (v5) at (1,1){};
\node[draw,circle,inner sep=2,fill=black] (v6) at (3,0){};

%graph edges
\path[draw,thick] (v1)--(v2);
\path[draw,thick] (v1)--(v3);
\path[draw,thick] (v2)--(v3);
\path[draw,thick] (v2)--(v4);
\path[draw,thick] (v3)--(v5);
\path[draw,thick] (v4)--(v5);
\path[draw,thick] (v4)--(v6);

\begin{scope}[very thick,decoration={
    markings,
    mark=at position 0.5 with {\arrow{>}}}
    ] 
    \draw[postaction={decorate},thick, myred,] (v2) to [max distance=10pt, out=135,in=215] (v3);
    \draw[postaction={decorate},thick, myred,] (v3) to [max distance=10pt, out=315,in=45] (v2);
    \draw[postaction={decorate},thick, myred,] (v2) to [max distance=10pt, out=45,in=135] (v4);
    \draw[postaction={decorate},thick, myred,] (v4) to [max distance=10pt, out=135,in=215] (v5);
\end{scope}
\end{scope}

\begin{scope}[shift={(4,0)}]
	
	%upper nodes for painting
	\node[circle,inner sep=0] (v1out) at (-3.5,0){};
	\node[circle,inner sep=0] (v3out) at (-1,1.5){};
	
	\node[circle,inner sep=0] (e12b) at (-2.2,-0.9){};
	\node[circle,inner sep=0] (e12t) at (-1.8,-0.1){};
	\node[circle,inner sep=0] (e13t) at (-2.2,0.9){};
	\node[circle,inner sep=0] (e13b) at (-1.8,0.1){};
	\node[circle,inner sep=0] (e23l) at (-1.4,0){};
	\node[circle,inner sep=0] (e23r) at (-.6,0){};
	\node[circle,inner sep=0] (e35t) at (0,1.4){};
	\node[circle,inner sep=0] (e35b) at (0,.6){};
	\node[circle,inner sep=0] (e24t) at (0,-.6){};
	\node[circle,inner sep=0] (e45l) at (.6,0){};

	%boundary nodes
	
	\node[circle,inner sep=0] (v2out) at (-1,-1.5){};
	\node[circle,inner sep=0] (v5out) at (1,1.5){};
	\node[circle,inner sep=0] (v3in) at (-1,.5){};
	
	\node[circle,inner sep=0] (e23a) at (-1.2,0){};
	\node[circle,inner sep=0] (e23b) at (-0.8,0){};
	\node[circle,inner sep=0] (e24) at (0,-1){};
	\node[circle,inner sep=0] (e45) at (1,0){};
	
	\path[draw=myred,thick] (v2out.center)..(e23a.center)..(v3in.center)..(e23b.center)..(e24.center)..(e45.center)..(v5out.center);

	%lower copies of boundary nodes

	\node[circle,inner sep=0] (v2outlow) at (-1,-1.5){};
	\node[circle,inner sep=0] (v5outlow) at (1,1.5){};
	\node[circle,inner sep=0] (v3inlow) at (-1,.5){};
	
	\node[circle,inner sep=0] (e23alow) at (-1.2,0){};
	\node[circle,inner sep=0] (e23blow) at (-0.8,0){};
	\node[circle,inner sep=0] (e24low) at (0,-1){};
	\node[circle,inner sep=0] (e45low) at (1,0){};
	
	\path[draw=myred,thick] (v2outlow.center)..(e23alow.center)..(v3inlow.center)..(e23blow.center)..(e24low.center)..(e45low.center)..(v5outlow.center);

	%lower nodes for painting
	
	\node[circle,inner sep=0] (v4out) at (1,-1.5){};
	\node[circle,inner sep=0] (v6out) at (3.5,0){};
	
	\node[circle,inner sep=0] (e24b) at (0,-1.4){};
	\node[circle,inner sep=0] (e45r) at (1.4,0){};
	\node[circle,inner sep=0] (e46b) at (2.2,-0.9){};
	\node[circle,inner sep=0] (e46t) at (1.8,-0.1){};

	%intersections in painting
	\path[draw,thick] (e12b.center)--(e12t.center);
	\path[draw,thick] (e13b.center)--(e13t.center);
	\path[draw,thick] (e23l.center)--(e23a.center);
	\path[draw,thick] (e23alow.center)--(e23blow.center);
	\path[draw,thick] (e23b.center)--(e23r.center);
	\path[draw,thick] (e45l.center)--(e45.center);
	\path[draw,thick] (e45low.center)--(e45r.center);
	\path[draw,thick] (e35b.center)--(e35t.center);
	\path[draw,thick] (e24b.center)--(e24low.center);
	\path[draw,thick] (e24.center)--(e24t.center);
	\path[draw,thick] (e46b.center)--(e46t.center);
	
	%painting
	\path[thick, draw,out angle=200, in angle=160] (e12b.center)..(v1out.center)..(e13t.center);
	\path[thick, draw,out angle=160, in angle=200] (e12t.center)..(e13b.center);

	\path[thick, draw,out angle=-60, in angle=170] (e12b.center)..(v2out.center);
	\path[thick, draw,out angle=-10, in angle=220] (v2outlow.center)..(e24b.center);
	\path[thick, draw,out angle=180, in angle=-90] (e24t.center)..(e23r.center);
	\path[thick, draw,out angle=-90, in angle=-20] (e23l.center)..(e12t.center);

	\path[name path=p1,thick, draw,out angle=60, in angle=-220] (e13t.center)..(v3out.center)..(e35t.center);
	\path[name path=p2,thick, draw,out angle=180, in angle=90] (e35b.center)..(e23r.center);
	\path[name path=p3, thick, draw,out angle=90, in angle=20] (e23l.center)..(e13b.center);

	\path[thick, draw,out angle=-120, in angle=-40] (e46b.center)..(v4out.center)..(e24b.center);
	\path[thick, draw,out angle=0, in angle=-90] (e24t.center)..(e45l.center);
	\path[thick, draw,out angle=-90, in angle=200] (e45r.center)..(e46t.center);
	
	\path[thick, draw,out angle=40, in angle=158] (e35t.center)..(v5out.center);
	\path[thick, draw,out angle=-22, in angle=40] (v5outlow.center)..(e45r.center);
	\path[thick, draw,out angle=90, in angle=0] (e45l.center)..(e35b.center);
	
	\path[thick, draw,out angle=70, in angle=-30] (e46t.center)..(v6out.center)..(e46b.center);
	
\end{scope}

\end{tikzpicture}
\caption{A walk $W$ in $G$ and a \cutline \ $f$ in $\mathcal{P}$ which induces the walk $W$ in $G$}\label{f:walk}
\end{figure}
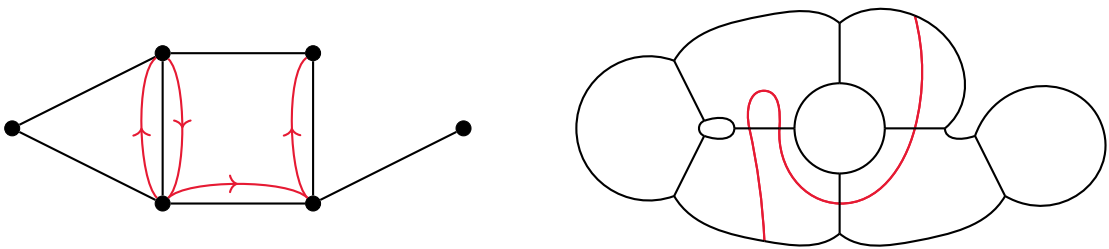

Let $f$ be a \cutline \ in $F_{\mathcal P}$ inducing a walk $W$ in $G_\mathcal{P}$. We define $\cut(\mathcal P,f) := \bigcup_{D \in \mathcal P} \mathcal C(D)$, where $\mathcal C(D)$ is the set of connected components of $\cut (D,f)$, where $\cut (D,f)$ is defined as in Definition~\ref{defn:pasting}. See Figures \ref{f:walk} and \ref{f:cut}.

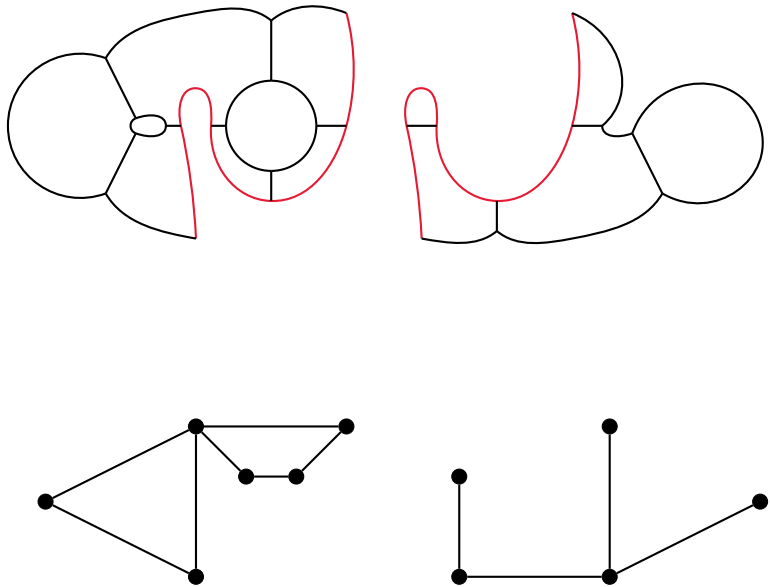
\begin{figure}[ht!]
	\begin{tikzpicture}[use Hobby shortcut]

	\begin{scope}[shift={(-1.5,0)}]
	%upper nodes for painting
	\node[circle,inner sep=0] (v1out) at (-3.5,0){};
	\node[circle,inner sep=0] (v3out) at (-1,1.5){};
	
	\node[circle,inner sep=0] (e12b) at (-2.2,-0.9){};
	\node[circle,inner sep=0] (e12t) at (-1.8,-0.1){};
	\node[circle,inner sep=0] (e13t) at (-2.2,0.9){};
	\node[circle,inner sep=0] (e13b) at (-1.8,0.1){};
	\node[circle,inner sep=0] (e23l) at (-1.4,0){};
	\node[circle,inner sep=0] (e23r) at (-.6,0){};
	\node[circle,inner sep=0] (e35t) at (0,1.4){};
	\node[circle,inner sep=0] (e35b) at (0,.6){};
	\node[circle,inner sep=0] (e24t) at (0,-.6){};
	\node[circle,inner sep=0] (e45l) at (.6,0){};
	\end{scope}
	
	%boundary nodes
	
	\begin{scope}[shift={(-1.5,0)}]
	\node[circle,inner sep=0] (v2out) at (-1,-1.5){};
	\node[circle,inner sep=0] (v5out) at (1,1.5){};
	\node[circle,inner sep=0] (v3in) at (-1,.5){};
	
	\node[circle,inner sep=0] (e23a) at (-1.2,0){};
	\node[circle,inner sep=0] (e23b) at (-0.8,0){};
	\node[circle,inner sep=0] (e24) at (0,-1){};
	\node[circle,inner sep=0] (e45) at (1,0){};
	
	\path[draw=myred,thick] (v2out.center)..(e23a.center)..(v3in.center)..(e23b.center)..(e24.center)..(e45.center)..(v5out.center);
	
	\end{scope}
	
	%lower copies of boundary nodes
	\begin{scope}[shift={(1.5,0)}]
	
	\node[circle,inner sep=0] (v2outlow) at (-1,-1.5){};
	\node[circle,inner sep=0] (v5outlow) at (1,1.5){};
	\node[circle,inner sep=0] (v3inlow) at (-1,.5){};
	
	\node[circle,inner sep=0] (e23alow) at (-1.2,0){};
	\node[circle,inner sep=0] (e23blow) at (-0.8,0){};
	\node[circle,inner sep=0] (e24low) at (0,-1){};
	\node[circle,inner sep=0] (e45low) at (1,0){};
	
	\path[draw=myred,thick] (v2outlow.center)..(e23alow.center)..(v3inlow.center)..(e23blow.center)..(e24low.center)..(e45low.center)..(v5outlow.center);
		
	\end{scope}

	%lower nodes for painting
	
	\begin{scope}[shift={(1.5,0)}]
	\node[circle,inner sep=0] (v4out) at (1,-1.5){};
	\node[circle,inner sep=0] (v6out) at (3.5,0){};
	
	\node[circle,inner sep=0] (e24b) at (0,-1.4){};
	\node[circle,inner sep=0] (e45r) at (1.4,0){};
	\node[circle,inner sep=0] (e46b) at (2.2,-0.9){};
	\node[circle,inner sep=0] (e46t) at (1.8,-0.1){};
	\end{scope}
	
	%intersections in painting
	\path[draw,thick] (e12b.center)--(e12t.center);
	\path[draw,thick] (e13b.center)--(e13t.center);
	\path[draw,thick] (e23l.center)--(e23a.center);
	\path[draw,thick] (e23alow.center)--(e23blow.center);
	\path[draw,thick] (e23b.center)--(e23r.center);
	\path[draw,thick] (e45l.center)--(e45.center);
	\path[draw,thick] (e45low.center)--(e45r.center);
	\path[draw,thick] (e35b.center)--(e35t.center);
	\path[draw,thick] (e24b.center)--(e24low.center);
	\path[draw,thick] (e24.center)--(e24t.center);
	\path[draw,thick] (e46b.center)--(e46t.center);
	
	%painting
	\path[thick, draw,out angle=200, in angle=160] (e12b.center)..(v1out.center)..(e13t.center);
	\path[thick, draw,out angle=160, in angle=200] (e12t.center)..(e13b.center);

	\path[thick, draw,out angle=-60, in angle=170] (e12b.center)..(v2out.center);
	\path[thick, draw,out angle=-10, in angle=220] (v2outlow.center)..(e24b.center);
	\path[thick, draw,out angle=180, in angle=-90] (e24t.center)..(e23r.center);
	\path[thick, draw,out angle=-90, in angle=-20] (e23l.center)..(e12t.center);

	\path[name path=p1,thick, draw,out angle=60, in angle=-220] (e13t.center)..(v3out.center)..(e35t.center);
	\path[name path=p2,thick, draw,out angle=180, in angle=90] (e35b.center)..(e23r.center);
	\path[name path=p3, thick, draw,out angle=90, in angle=20] (e23l.center)..(e13b.center);

	\path[thick, draw,out angle=-120, in angle=-40] (e46b.center)..(v4out.center)..(e24b.center);
	\path[thick, draw,out angle=0, in angle=-90] (e24t.center)..(e45l.center);
	\path[thick, draw,out angle=-90, in angle=200] (e45r.center)..(e46t.center);
	
	\path[thick, draw,out angle=40, in angle=158] (e35t.center)..(v5out.center);
	\path[thick, draw,out angle=-22, in angle=40] (v5outlow.center)..(e45r.center);
	\path[thick, draw,out angle=90, in angle=0] (e45l.center)..(e35b.center);
	
	\path[thick, draw,out angle=70, in angle=-30] (e46t.center)..(v6out.center)..(e46b.center);

\begin{scope}[shift={(0,-5))}]

	\begin{scope}[shift={(-1.5,0)}]
	\node[draw,circle,inner sep=2,fill=black] (v1) at (-3,0){};
\node[draw,circle,inner sep=2,fill=black] (v2) at (-1,-1){};
\node[draw,circle,inner sep=2,fill=black] (v3) at (-1,1){};
\node[draw,circle,inner sep=2,fill=black] (v4) at (1,1){};
\node[draw,circle,inner sep=2,fill=black] (v5) at (-1/3,1/3){};
\node[draw,circle,inner sep=2,fill=black] (v6) at (1/3,1/3){};

%graph edges
\path[draw,thick] (v1)--(v2);
\path[draw,thick] (v1)--(v3);
\path[draw,thick] (v2)--(v3);
\path[draw,thick] (v3)--(v4);
\path[draw,thick] (v3)--(v5);
\path[draw,thick] (v5)--(v6);
\path[draw,thick] (v6)--(v4);
	\end{scope}
	
	\begin{scope}[shift={(2,0)}]
	
	\node[draw,circle,inner sep=2,fill=black] (v2) at (-1,-1){};
	\node[draw,circle,inner sep=2,fill=black] (v3) at (-1,1/3){};
	\node[draw,circle,inner sep=2,fill=black] (v4) at (1,-1){};
	\node[draw,circle,inner sep=2,fill=black] (v5) at (1,1){};
	\node[draw,circle,inner sep=2,fill=black] (v6) at (3,0){};

%graph edges
\path[draw,thick] (v2)--(v3);
\path[draw,thick] (v2)--(v4);
\path[draw,thick] (v4)--(v5);
\path[draw,thick] (v4)--(v6);
	\end{scope}
	
	\end{scope}
	\end{tikzpicture}
\caption{The painting $\cut(\mathcal{P},f)$ and the canonical graph $G_{\cut(\mathcal{P},f)}$.}\label{f:cut}
\end{figure}

\begin{lemma}
\label{lem:cutpainting}
Let $\mathcal P$ be a painting, and let $f$ be a \cutline \ in $F_{\mathcal P}$ inducing a non-trivial walk, then $\cut (\mathcal P, f)$ is a painting.
\end{lemma}
\begin{proof}
Since $f$ induces a non-trivial walk, its image intersects any $D \in \mathcal P$ in a finite number (potentially 0) of arcs. This implies that $\mathcal C(D)$ consists of finitely many disjoint closed disks, which are contained in the surface $\cut( F_{\mathcal P}, f)$.

For $i \in \{1,2,3\}$, let $E_i \in \cut(\mathcal P,f)$, and let $D_i \in \mathcal P$ such that $E_i \in \mathcal C(D_i)$. 
If $D_1$, $D_2$, and $D_3$ are distinct, then $E_1 \cap E_2 \cap E_3 = \emptyset$ due to $\mathcal P$ being a painting.
If $D_i = D_j$ for some $i\neq j$, then $E_1 \cap E_2 \cap E_3 = \emptyset$ since the elements of $\mathcal C(D_i)$ are disjoint. 

Next assume that $E_1 \cap E_2 \neq \emptyset$. In this case $D_1 \neq D_2$, and  $D_1 \cap D_2 \neq \emptyset$. Recall that the image of $f$ intersects $D_1 \cap D_2$ in finitely many points, all of which are interior to some interval in $D_1 \cap D_2$. Consequently  $\cut (D_1 \cap D_2, f)$ consists of finitely many disjoint arcs. Each of these arcs is connected and thus either fully contained in both $E_1$ and $E_2$, or disjoint from at least one of them. In particular $E_1 \cap E_2$ consists of a collection of disjoint arcs.

Since $E_1$, $E_2$, and $E_3$ were arbitrary, this shows that $\cut(\mathcal P, f)$ is a painting.
\end{proof}

\begin{lemma}
\label{lem:cutpainting-uniquedisk}
Let $\mathcal P$ be a painting, and let $f$ be a \cutline \ in $F_{\mathcal P}$ inducing a non-trivial walk. Let $z\in F_{\cut(\mathcal P,f)}$, and let $\phi\colon F_{\cut(\mathcal P,f)} \to F_{\mathcal P}$ denote the pasting map\footnote{c.f.\ Definition \ref{defn:pasting}, and note that  $\cut(F_{\mathcal P},f) = F_{\cut(\mathcal P,f)}$}.
If $\phi(z)$ is contained in a unique disk $D \in \mathcal P$, then $z$ is contained in a unique disk $E \in \cut(\mathcal P,f)$.
\end{lemma}

\begin{proof}
Assume that there are distinct elements $E,E' \in \cut(\mathcal P,f)$ such that $z \in E \cap E'$. Since different elements of $\mathcal C(D)$ are disjoint for any $D$, we know that there are different elements $D,D' \in \mathcal P$ such that $E\in \mathcal C(D)$ and $E' \in \mathcal C(D')$. By definition of $\phi$ we have $\phi(E \cap E') \subseteq D \cap D'$, thus $\phi(z) \in D\cap D'$ which proves the lemma.
\end{proof}

\begin{lemma}
\label{lem:cutpainting-homomorphism}
Let $\mathcal P$ be a painting, and let $f$ be a \cutline \ in $F_{\mathcal P}$ inducing a non-trivial walk $W$. Then there is a painting homomorphism $\phi\colon \cut (\mathcal P, f) \to \mathcal P$ with the following properties: 
\begin{itemize}
    \item for $D, D' \in \mathcal P$ with $D \cap D' \neq \emptyset$ there are $E \in \phi^{-1}(D)$, and $E' \in \phi^{-1}(D')$ with $E \cap E' \neq \emptyset$, and 
    \item the restriction of $\phi$ to $\phi^{-1}(\mathcal P - W)$ is an isomorphism onto $\mathcal P - W$.
\end{itemize}
\end{lemma}

\begin{proof}
Consider the map $\phi:\cut(\mathcal P, f) \to \mathcal P$, mapping all of $\mathcal C(D)$ to $D$. Clearly, if $E $ and $E'$ touch then $\phi(E)$ and $\phi(E')$ are distinct and $\phi(E) \cap \phi(E') \neq \emptyset$, whence $\phi$ is a homomorphism as desired. 
If $D \cap D' \neq \emptyset$ for $D,D' \in \mathcal P$, then for any point $x \in D \cap D'$ there are $E \in \mathcal C(D)$, and $E' \in \mathcal C(D')$ containing $x$, thus showing the first claimed property of $\phi$.
For $D \in \mathcal P \setminus W$, we know that $D$ does not meet the image of $f$, and thus $\mathcal C(D) = \{D\}$. Thus the restriction of $\phi$ on $\phi^{-1}(G_{\mathcal P}-W)$ is the identity map, and in particular an isomorphism. 
\end{proof}

\begin{theorem}\label{t:strat}
Let $\mathcal P$ be a painting, then $c(G_{\mathcal P}) \leq v(F_{\mathcal P})+1$.
\end{theorem}
\begin{proof}
We play the cops and robber game on $G_{\mathcal P}$ and the topological Marker-Cutter game on $F_{\mathcal P}$ in parallel. Game play proceeds in \emph{stages}, where each stage consists of a turn of Marker, followed by finitely many turns in the cops and robber game, and finally a turn of cutter. We will assume that Marker plays optimally. The cops' strategy depends on Marker's move, and Cutter's move depends on the outcome of the cops and robber game in the current stage.

First note that if $\mathcal Q$ is a painting and $f$ is a \cutline \ in $F_{\mathcal Q}$ inducing a non-trivial walk, then $\cut(F_{\mathcal Q}, f) = F_{\cut(\mathcal Q,f)}$. Note that connected components of $\cut(F_\mathcal Q,f)$ correspond to components of the graph $G_{\cut(\mathcal Q,f)}$. We will ensure that in every stage of the game the surface $S_k$ is $F_{\mathcal{Q}}$ for some graph painting $\mathcal{Q}$ and the \cutline \ $f_{k+1}$ induces a non-trivial walk in the graph $G_{\mathcal Q}$.

In particular, we can throughout the game assume that $S_k = F_{\mathcal P_k}$, where $\mathcal P_0 = \mathcal P$ and $\mathcal P_{k+1}$ is the vertex set of a component of $G_{\cut(\mathcal P_k, f_{k+1})}$. Lemma \ref{lem:cutpainting} gives a painting homomorphism $\phi_k\colon \mathcal P_{k+1} \to \mathcal P_k$ which is an isomorphism onto its image when restricted to elements that have empty intersection with $A_{k+1} = f_{k+1}([0,1])$. We define painting homomorphisms $\psi_k\colon \mathcal P_k \to \mathcal P$ by $\psi_0 = \id$, and $\psi_{k+1} = \psi_k \circ \phi_k$. It is easy to see that $D \in \mathcal P_k$ is contained in the image of $\phi_k$ if and only if $D \cap F_{\mathcal P_{k+1}} \neq \emptyset$, and by induction, $D \in \mathcal P$ is in the image of $\psi_k$ if and only if $D \cap F_{\mathcal P_{k}} \neq \emptyset$.

We say that $D \in \mathcal P_k$ is \emph{covered} by an active index $i$, if $\partial D \cap (A_i \cup A_i') \neq \emptyset$, and \emph{uncovered} if it is not covered by any active index. Let $\tilde{\mathcal C}_k(i)$ be the set of elements covered by $i$, and let $\tilde {\mathcal R}_k$ be the set of uncovered elements. Let $\tilde G_k(i)$ be the subgraph of $G_{\mathcal P_k}$ induced by $\tilde{\mathcal C}_k(i) \cup \tilde {\mathcal R}_k$. Let $\mathcal C_k(i) = \psi_k(\tilde{\mathcal C}_k(i))$, let $\mathcal R_k = \psi_k (\tilde {\mathcal R}_k)$, and let $G_k(i) = \psi_k (\tilde G_k(i))$, where $\psi_k$ is interpreted as a graph homomorphism $G_{\mathcal P_k} \to G_{\mathcal P}$. Note that the vertex set of $G_k(i)$ is $\mathcal C_k(i) \cup \mathcal R_k$, but it is not the induced subgraph on this set, since $\psi_k^{-1}(\mathcal C_k(i))$ might be a strict superset of $\tilde{\mathcal C}_k(i)$, and thus it is possible that not every edge whose endpoints are in $\mathcal C_k(i) \cup \mathcal R_k$ has a preimage in $\tilde G_k(i)$.

Observe that for any $R \in \mathcal R_k$, the disk $\psi_k^{-1}(R)$ is fully contained in $F_{\mathcal P_k}$. In particular, if $D$ is a neighbour of $R$ in $G_{\mathcal P}$, then part of $D$ (namely $\partial D \cap \partial R$) is also contained in $F_{\mathcal P_k}$, and there is $\tilde D \in \psi_k^{-1}(D)$ such that $\partial D \cap \partial R \subset \tilde D$. Either $\tilde D$ is uncovered, in which case $\tilde D \in \mathcal R_k$, or $\tilde D$ is covered by some active index $i$. Hence, any edge incident to $\mathcal R_k$ is contained in $G_{k}(i)$ under $\psi_k$ for some active index $i$. 

If an element $D$ of $\mathcal P_{k+1}$ is covered by an active index $i < k+1$ in stage $(k+1)$, then $\phi_k(D)$ is covered by $i$ in stage $k$. Consequently, $\mathcal C_{k+1}(i) \subseteq \mathcal C_{k}(i)$ and $G_{k+1}(i)$ is a subgraph of $G_k(i)$.
%Since also the robber territory gets smaller.

Conversely note that for $D \in \mathcal P_{k+1}$, if $\phi_k(D)$ is covered by $i$ in stage $k$, and $D$ does not have $A_i$ or $A_i'$ in its boundary, then $D$ is a proper subset of $\cut(\phi_k(D),f_{k+1})$ and thus must have $A_{k+1}$ or $A_{k+1}'$ in its boundary. Hence if $\phi_k(D)$ is covered by $i$ in stage $k$, then $D$ is covered by at least one of the indices $i$ and $(k+1)$ in stage $(k+1)$, possibly by both. 

Hence, $\phi_k$ maps uncovered vertices to uncovered vertices. Induction combined with the fact, from Lemma \ref{lem:cutpainting-homomorphism}, that the restriction of $\phi_k$ to elements not covered by index $(k+1)$ is an isomorphism onto its image shows that the restriction of $\psi_k$ to the uncovered vertices is an isomorphism onto $\mathcal R_k$. This also shows that $\mathcal R_k$ is disjoint from $\mathcal C_k(i)$ for any active index $i$.

We will inductively make sure that after every stage in which the robber is not caught the following properties are satisfied.
\begin{enumerate}[label=(\arabic*)]
    \item \label{itm:smallerterritory} $\mathcal R_k$ is a strict subset of $\mathcal R_{k-1}$,
    \item \label{itm:robberinterritory} the robber is in $\mathcal R_k$,
    \item \label{itm:coveredguarded} for every active index $i$, there is a cop $c_i$ guarding $\mathcal C_k(i)$ in some subgraph of $G_{\mathcal P}$ containing $G_k(i)$.
\end{enumerate}
Furthermore, we inductively guarantee the following technical condition which will make sure that in the topological game $f_k$ can be chosen to induce a walk:
\begin{enumerate}[resume,label=(\arabic*)]
    \item \label{itm:Xunique} any element of $X_k$ is contained in a unique disk in $\mathcal P_k$.
\end{enumerate}

Before the first stage, with $\mathcal R_0 = \mathcal P$, $X_0 = \emptyset$ and no active indices \ref{itm:smallerterritory}-- \ref{itm:Xunique} are trivial. Since $\mathcal R_0$ is finite, the induction on \ref{itm:smallerterritory} and \ref{itm:robberinterritory} cannot continue indefinitely, in particular the robber must be caught in some stage $n \leq |\mathcal P|$. Furthermore, we will see that we can complete each stage using a single cop in addition to the cops $c_i$ from \ref{itm:coveredguarded}, thus proving the theorem.

We now describe the strategy in stage $k$. Let $x,y$ be the two points chosen by Marker, and let $D_x$ and $D_y$ be the elements of $\mathcal P_k$ containing $x$ and $y$ respectively. We can assume that $D_x$ and $D_y$ are uniquely determined. If $x$ and $y$ are in $X_k$ this follows from \ref{itm:Xunique}, if they are interior points of $F_{\mathcal P _k}$, then we can apply a small perturbation without changing the effect of Marker's move. In $G_{\mathcal P_k}$, pick a non-trivial walk $W$ from $D_x$ to $D_y$ which contains some uncovered element and is shortest possible subject to these restrictions. 

Let $H$ be the subgraph of $G_{\mathcal P_k}$ induced by $W \cup \tilde{\mathcal R}_k$. We claim that $\psi_k(W)$ is guardable relative to $\mathcal R_k$ in $\psi_k(H)$---we remark, that a subgraph of this image will correspond to $G_{k+1}(k+1)$. By Lemma \ref{lem:relativeguard} it is sufficient to show that $W$ is guardable relative to $\tilde{\mathcal R}_k$ in $H$. If $|V(W)|=2$, then it is clear.

Otherwise consider an auxiliary graph $G$ which is constructed as follows: start with a copy $R$ of the graph induced by $\tilde{\mathcal R}_k$, and a path $P$ of the same length as $W$, add edges from the $r$-th vertex of $P$ to the vertices of $R$ corresponding to neighbours of the $r$-th vertex of $W$ in $\tilde{\mathcal R}_k$. Finally, if the $r$-th vertex of $W$ is in $\tilde{\mathcal R}_k$, then identify the $r$-th vertex of $P$ with the corresponding vertex of $R$, and denote the resulting graph by $G$. 

Let $\varphi \colon G \to H$ be the map sending the $r$-th vertex of $P$ to the $r$-th vertex of $W$, and mapping every vertex in $R$ to the corresponding vertex in $\tilde{\mathcal R}_k$. Clearly, $f$ is a homomorphism and its restriction to $R = \varphi^{-1}(\tilde{\mathcal R}_k)$ is an isomorphism onto its image. By construction, every edge incident to $\tilde{\mathcal R}_k$ is the image of some edge of $G$ under $\varphi$. Hence, again by Lemma \ref{lem:relativeguard}, we only need to show that $P$ is guardable relative to $R$ in $G$.
If $P$ is not a shortest path between its endpoints in $G$, then let $P'$ be a shorter path. Via $\varphi$, the path $P'$ induces a walk $W'$ from $D_x$ to $D_y$ of the same length as $P'$. Since $P - \varphi^{-1}(\tilde{\mathcal R}_k)$ is not connected, $W'$ would contain an uncovered vertex thus contradicting the minimality of $W$ unless $W'$ is a trivial walk. However in this case $D_x = D_y$ is uncovered and hence $|V(W)| =2$. Hence $P$ is a shortest path between its endpoints and thus by Lemma \ref{lem:guardshortestpath} it is guardable (and consequently guardable relative to $R$) in $G$.

The cops' strategy in stage $n$ is now as follows. For every active index $i$, the cop $c_i$ keeps guarding $\mathcal C_k(i)$ in some graph containing $G_k(i)$. Since any edge connecting $\mathcal R_k$ to its complement is contained in some $G_k(i)$ this makes sure that robber cannot leave $\mathcal R_k$ without being caught. An additional cop $c_{k+1}$ plays such that after finitely many steps she guards $\psi_k(W)$ in $\psi_k(H)$.

For Cutters move, note that in the above argument we can choose the walk $W$ as follows. Let $D$ be an uncovered element in $W$ and let $P_x$ and $P_y$ be shortest paths from $D$ to $D_x$ and $D_y$ respectively, chosen in such a way that they have the largest possible number of common vertices. The walk $W$ obtained by following $P_x$ from $D_x$ to $D$ and then $P_y$ from $D$ to $D_y$ induces a tree with at most three leaves ($D_x$, $D_y$ and possibly $D$), and since $D_x$ and $D_y$ are the only disks in $\mathcal P_k$ containing $x$ and $y$ respectively, it is easy to see that there is a \arc{x}{y} $f_{k+1}$ in $F_{\mathcal P_k}$ which induces $W$. Cutter plays this arc. If there is more than one component in $\cut(F_{\mathcal P_k},f_{k+1})$, then let $R \in \mathcal R_k$ be the position of the robber after finishing the cops-and-robber moves of this stage. If $A_{k+1}$ passed through $\psi_k^{-1}(R)$, then the robber would have been caught. Consequently, $\psi_k^{-1}(R)$ completely lies in one component of $\cut(F_{\mathcal P_k},f_{k+1})$, and this is the component that Cutter chooses to become the active component. 

Any cop $c_i$ for which $i$ has become inactive due to this is released from her guarding duty and consequently can be reassigned to guard a different graph in subsequent rounds, thus bounding the total number of cops used throughout the strategy by $v(F_{\mathcal P})+1$.

It remains to check that our inductive assumptions \ref{itm:smallerterritory}--\ref{itm:Xunique} hold for the next stage. 

Since $\phi_k$ maps uncovered vertices (after stage $k+1$) to uncovered vertices (after stage $k$) we know that $\mathcal R_{k+1} \subset \mathcal R_k$. It is a strict subset because $W$ contained at least one uncovered (after stage $k$) element $D$, which is now covered by index $(k+1)$. Hence $\psi_k(D)$ is contained in $\mathcal R_k$, but not in $\mathcal R_{k+1}$, and we have proved \ref{itm:smallerterritory}. 

If the robber tries to leave $\mathcal R_k$ at any point during stage $k+1$, he would have been caught, since any edge connecting $\mathcal R_k$ with its complement is contained in some $G_k(i)$. He also cannot be at any vertex of $W$ at the end of stage $k+1$, or he would have been caught by $c_{k+1}$. Hence he must be at a vertex in some component of $\mathcal R_k - W$, and the choice of the active component made by Cutter ensures that this is $\mathcal R_{k+1}$, thus proving \ref{itm:robberinterritory}. 

For \ref{itm:coveredguarded}, first note that for any active index $i < k+1$ we have $\mathcal C_{k+1} (i) \subseteq \mathcal C_k(i)$ and $G_{k+1}(i) \subseteq G_k(i)$, so cop $c_i$ can keep playing the same strategy. cop $c_{k+1}$ guards $\psi_k(W)$ in $\psi_k(H)$. Note that a vertex $D \in \mathcal P_{k+1}$ is covered by $(k+1)$ if and only if $\phi_k(D)$ is contained in $W$, whence $\psi_{k+1}(\tilde{\mathcal C}_{k+1}(k+1)) \subseteq \psi_k(W)$, and that $\phi_k$ maps uncovered vertices to uncovered vertices, whence $\psi_{k+1}(\tilde{\mathcal R}_{k+1}) \subseteq \psi_k(\tilde{\mathcal R}_k)$. It follows that $\psi_k(W)$ contains $\mathcal C_{k+1}(k+1)$ and $\psi_k(H)$ contains $G_{k+1}(k+1)$ which completes the proof of \ref{itm:coveredguarded}.

Finally, \ref{itm:Xunique} holds by Lemma \ref{lem:cutpainting-uniquedisk} since any element of $X_k \cup \{x,y\}$ is contained in a unique disk of $\mathcal P_k$.
\end{proof}

\section{The combinatorial Marker-Cutter game}\label{s:CombGame}
Since we only consider surfaces up to homeomorphism, the only information that will inform Cutter and Marker's strategy at each turn will be the combinatorial information contained in the way the active arcs are arranged in the components of the boundary, as well as the number of handles in the active surface.

Using this, we can relate the topological Marker-Cutter game to a yet another game, which we call the \emph{combinatorial Marker-Cutter game}. The benefit being that it will be much simpler to describe the `state' of this game at each step, which will help in describing our strategy for Marker.

\begin{defn}
A set of \emph{boundary cycles} is a pair $(D,\chi)$ of a finite directed graph $D$ each of whose components is a directed cycle (loops are allowed, but not isolated vertices), together with an edge-labelling $\chi \colon E(D) \rightarrow L$. In a standard abuse of notation, we shall often use just $D$ to refer to the set of boundary cycles. A set of boundary cycles in which each label is used at most twice is called \emph{proper}. A label which only appears on a single cycle in $D$ is called \emph{isolated}, and it is called \emph{uniquely appearing} if it only appears on a single edge. The set of {\em active labels} $A_{D,\chi} = \chi(E(D))$ is the set of labels appearing on the edges of $D$, and the {\em value} $v(D,\chi)$ is defined as $|A_{D,\chi}|$.

%Let us say the label of a $\partial$-arc $A_i$ or $A'_i$ in $\mathcal{A}_k$ is $i$. For each component of $\partial S_k$ which contains $i$ $\partial$-arcs from $\mathcal{A}_k$, the orientation of the surface determines a cyclic order on the labels of the $\partial$-arcs from $\mathcal{A}_k$ in this component, which gives in a natural way a cyclic sequence of length $i$. Let $\mathcal{D}$ be the union of the cyclic sequences corresponding to each component of $\partial S_k$, we say that $\mathcal{D}$ is a \emph{representation} of the game state given by the partial play $(A_i, S_i \colon i \in [k])$.

A \emph{game state} in the combinatorial Marker-Cutter game is a triple $\gamma=(D,\chi,g)$ where $(D,\chi)$ is a proper set of boundary cycles, along with a counter $g \in \mathbb{N}_0$. The \emph{value} $v(\gamma)$ of a game state $\gamma$ is just the value $v(D,\chi)$ of its set of boundary cycles.
\end{defn}

We start the combinatorial Marker-Cutter game with $\gamma_0=(D_0,\chi_0,g_0)$ where $D_0$ is the empty digraph, $\chi_0= \emptyset$ and $g_0 \in \mathbb{N}$. In a general turn we have a game state $(D_k,\chi_k,g_k)$ and first Marker, then Cutter makes a move.

\textbf{Marker's move.} Marker chooses two elements $v,w$ from $V(D_k) \, \dot\cup \, \{v_\gamma,w_\gamma\}$, not necessarily distinct. The resulting tuple $(D_k,\chi_k,g_k,v,w)$ is also called a \emph{marked game state}. The vertices $\{v_\gamma,w_\gamma\}$ signify that Marker may also choose up to two `dummy' vertices not in $V(D_k)$ (or a single `dummy' vertex twice). If $v$ or $w$ is a dummy vertex then we add a new vertex to $D_k$ for each dummy vertex chosen (without multiplicity). Let $C,C' \subseteq D_k$ be the components containing $v$ and $w$ respectively, so that both $C$ or $C'$ are either a cycle or a single dummy vertex.

\bigskip
\textbf{Cutter's move.}

\begin{enumerate}[label=]
\item \textbf{Case $\mathbf{C = C'}$:} By splitting the vertex $v$ into a source $v_1$ and a sink $v_2$, and similarly for $w$, the directed cycle $C$ gives rise to two directed paths $P$ from $v_1$ to $w_2$ and $P'$ from $w_1$ to $v_2$. If $v=w$ is a dummy vertex, then we let $v_1=w_2$ and and $v_2=w_1$ and these paths are trivial. Consider two new directed edges $f = \overrightarrow{w_2v_1}$ and $f'=\overrightarrow{v_2w_1}$, and define
$$\hat{C}_1 = \hat{C}_1[C,v,w] := P \cup \{f\} \; \text{ and } \; \hat{C}_2 = \hat{C}_2[C,v,w] := P' \cup \{f'\}.$$
So, if $v=w$ is a dummy vertex, then $\hat{C}_1 = f$ and $\hat{C}_2 = f'$ are loops on $v_1$ and $v_2$ respectively.
%(Josh): I changed the above, but I think its correct now.

Then Cutter makes one of the following choices for the next game state:
\begin{enumerate}
\item \label{movea} $D_{k+1}  = \overline{D} \cup \hat{C}_1 \cup \hat{C}_2$ and $g_{k+1} = g_k-1$ (provided $g_k \geq 1$);
%(Josh) : Shouldn't let Cutter delete bits in type (a) moves
\item \label{moveb} $D_{k+1}'  = \overline{D} \cup \hat{C}_1$ and $g_{k+1} = \overline{g}$; or
\item \label{movec} $D_{k+1}  = \overline{D} \cup \hat{C}_2$ and $g_{k+1} = \overline{g}$,
\end{enumerate}
where $\overline{D}$ is any union of components of $D_k \setminus C$, $0 \leq \overline{g} \leq g_k$, $\chi_{k+1} = \chi_k$ wherever it is defined, and $\chi_{k+1}(f) = \chi_{k+1}(f')= \ell'$ (if they exist) where $\ell'$ is a new label outside the image of $\chi_k$. Note that we do not allow Cutter to make the first choice if $g_k=0$.
\bigskip

\item \textbf{Case $\mathbf{C \neq C'}$:} By splitting the vertex $v$ into a source $v_1$ and a sink $v_2$, and similarly for $w$, the directed cycles $C$ and $C'$ give rise to two directed paths $P$ from $v_1$ to $v_2$ and $P'$ from $w_1$ to $w_2$. If $v$ is a dummy vertex then we take $v_1=v_2$ so that $P$ is a trivial path, and similarly for $w$. Consider two new directed edges $f = \overrightarrow{w_2v_1}$ and $f'=\overrightarrow{v_2w_1}$ and define the \emph{amalgamated cycle} $\hat{C} =\hat{C}((C,v),(C',w))$ as $\hat{C}:= P \cup P' \cup \{f,f'\}$.

In this case, Cutter has no choice, and the next game state is:
\begin{enumerate}[resume]
  \item \label{moved}$D_{k+1}=( D_k\cup \hat{C}) \setminus (C \cup C')$ and  $g_{k+1}=g_k$,
\end{enumerate}
$\chi_{k+1} = \chi_k$ wherever it is defined, and $\chi_{k+1}(f) = \chi_{k+1}(f')= \ell'$ where $\ell'$ is a new label outside the image of $\chi_k$. 
\end{enumerate}

Note that in all cases if $(D_k,\chi_k)$ is proper, then so is $(D_{k+1},\chi_{k+1})$, and hence $(D_k,\chi_k,g_k)$ is a game state. 

By Lemma \ref{l:boundaryprop}, every partial play $(f_i, S_i \colon i \in [k])$ in the topological Marker-Cutter game determines a game state in the combinatorial Marker-Cutter game in the following way:

For each component of $\partial S_k$ which contains $m \geq 1$ $\partial$-arcs from $\mathcal{A}_k$, the orientation of the surface determines a cyclic order on the labels of the $\partial$-arcs from $\mathcal{A}_k$ in this component, which corresponds in a natural way a directed cycle of length $m$, whose edges correspond to $\partial$-arcs. Let $D$ be the digraph given by taking a disjoint union of one such directed cycle for each component of $\partial S_k$ containing at least one $\partial$-arc from $\mathcal{A}_k$. We can then define an edge-labelling $\chi$ on $E(D)$ by giving an edge corresponding to a $\partial$-arc $A_i$ or $A'_i$ in $\mathcal{A}_k$ the label $i$. Finally let $g=g(S_k)$. We say that the game state $(D,\chi,g)$ is a \emph{representation} of the partial play $(f_i, S_i \colon i \in [k])$. Note that if $(D,\chi,g)$ is a representation of $(f_i, S_i \colon i \in [k])$, then the number of active indices $a_k$ in $S_k$ is the value of the game state $v(\gamma)$.

\begin{lemma}\label{l:gameequiv}
Let $S$ be a compact connected orientable surface $S$. If Marker has a strategy in the combinatorial Marker-Cutter game with $g_0 = g(S)$ such that the value of each game state is bounded by $t$, then $v(S) \leq t$.
\end{lemma}
\begin{proof}
Suppose Marker has such a strategy in the combinatorial game. We will use our strategy for Marker in the combinatorial Marker-Cutter game to give a strategy for Marker in the topological Marker-Cutter game by playing the two games `in parallel'. Suppose that both the combinatorial game and topological game have been played for $k$ turns, resulting in a game state $(D_k,\chi_k,g_k)$ which is a representation of the partial play $(f_i, S_i \colon i \in [k])$. 

At the start of the topological Marker-Cutter game on $S$ we have $S_0 = S$ and $X_0 = \emptyset$ and at the start combinatorial Marker-Cutter game the game state is $(D_0,\chi_0,g_0)$ where $D_0$ is the empty digraph, $\chi_0$ the empty function, and $g_0=g(S)$, and so $(D_0,\chi_0,g_0)$ is a representation of the empty partial play.

Combinatorial Marker's strategy in the combinatorial game calls for him to choose two elements $v,w$ from $V(D_k) \, \dot\cup \, \{v_\gamma,w_\gamma\}$. If $v \in \{v_\gamma,w_\gamma\}$ then our strategy for topological Marker is to choose a point $x \in \Int{S_k}$, and if $w \in \{v_\gamma,w_\gamma\}$ topological Marker chooses a point $y \in \Int{S_k}$ where $x=y$ if and only if $v=w$. Otherwise, say $v$ is a vertex on some cycle $C$ of $D_k$. Since $(D_k,\chi_k,g_k)$ is a representation of the partial play $(f_i, S_i \colon i \in [k] )$, the two edges incident to $v$ in $D_k$ correspond to $\partial$-arcs $B$ and $B'$ in $\mathcal{A}_k$, where we take the convention that $B=B'$ if $v$ is vertex in a loop. Then $B$ and $B'$ lie in the boundary of the same disk $\partial D \subseteq \partial S_k$ and furthermore, one of the two components of $D \setminus (\Int{B} \cup \Int{B}')$ contains no active arc. Let $B''$ be this component. Note that $B''$ is a (perhaps trivial) arc, and $\partial B'' \subseteq \partial B  \cup \partial B' \subseteq X_k$ by Lemma~\ref{l:boundaryprop}, where we let $\partial B'' = B''$ if $B''$ is a single point. In this case, our strategy for topological Marker is to choose a point $x \in \partial(B'') \subseteq X_k$. Marker makes a similar choice for $y$ if $w \in V(D_k)$.\footnote{Note that, even if $v=w \in V(C)$, technically Marker could choose different points $x$ and $y$. We could insist that Marker makes these choices in some consistent manner, but it will not matter.}

Topological Cutter then chooses an \arc{x}{y} $f_{k+1}$ in $S_k$ and a component $S_{k+1}$ of Cut$(S_k,f_{k+1})$. Note that, by Lemmas \ref{l:x=yorC=C'} and \ref{l:xneqyandCneqC'}, the representation $(D_{k+1},\chi_{k+1},g_{k+1})$ of the partial play $(f_i, S_i \colon i \in [k+1])$ can be obtained in the combinatorial game from the marked game $(D_k,\chi_k,g_k,v,w)$ by making one of the four moves available to combinatorial Cutter in the combinatorial Marker-Cutter game, and combinatorial Cutter makes that move.

Hence if we follow the above strategy, at each turn of the game combinatorial Marker plays according to his strategy, and $(D_k,\chi_k,g_k)$ is a representation of the partial play $(f_i, S_i \colon i \in [k] )$. Hence, by the comment above the theorem, the number of active indices $a_k$ is equal to the value of the game state $v(\gamma)$ for each $k$. Since combinatorial Marker is following a strategy which bounds the value of each game state by $t$, it follows that the score of this play of the topological game is $\sup_{ k \in \mathbb{N}} a_k \leq t$ and hence $v(S) \leq t$.
\end{proof}

The converse of Lemma \ref{l:gameequiv} is also true, and so the combinatorial and topological Marker-Cutter games are equivalent. However we will not use this fact, and so omit the details.

\subsection{Restricted Combinatorial Cutter}
In the combinatorial Marker-Cutter game, Cutter may often make moves that either do not change the game state in any meaningful way, or even make it `simpler' (and in this way worse for Cutter). 

It will be useful for us to show that we may assume that Cutter, when playing optimally, does not make any such moves. This will allow us to control the number of options that Cutter has at each stage, and so explicitly describe a strategy of Marker which keeps the value of the game state low.

\begin{defn}
\label{def_equivalence}
Given a set of (proper) boundary cycles $(D,\chi)$, a \emph{contraction} is performed by choosing an edge $e \in E(D)$ and contracting it (deleting any resulting isolated vertices). 

A game state $(D',\chi',g')$ is a \emph{reduction} of a game state $(D,\chi,g)$ if $g'\leq g$, $\chi' = \chi$ on $E(D')$ and $D$ can be obtained from $D'$ via a sequence of contractions.

Two game states $(\mathcal{D},\chi,g)$ and $(\mathcal{D}',\chi',g)$ are \emph{equivalent} if there is a bijection $\phi$ from the labels of $\chi$ to the labels of $\chi'$ such that $\phi \circ \chi = \chi'$. 

We write $(D',\chi',g') \preccurlyeq (D,\chi,g)$ if $(D',\chi',g')$ is equivalent to a reduction of $(D,\chi,g)$.
\end{defn}

\begin{defn}
\emph{Restricted Cutter} is a player who is never allowed to make a move which returns to a game state which is equivalent to a reduction of an earlier game state, and is always forced to take $\overline{D} = D \setminus C$ and $\overline{g}=g$ for a move of type (a), (b) or (c). If at any point restricted Cutter has no legal moves, the game ends. \end{defn}

One might expect that an optimal strategy for combinatorial Cutter should never move to a game state which is equivalent to a reduction of the game state at a previous step, since any value he can achieve from the reduced game state, he should also have been able to achieve from the game state at the previous step, with perhaps even the advantage of some extra edges which were contracted in the reduction. Our next theorem justifies this intuition. In this theorem, we say a game state $(D,\chi,g)$ \emph{occurs} in a game played according to some strategy of Marker, if there is a sequence of moves for Marker and Cutter, where Marker always plays according to his strategy, which eventually arrives at $(D,\chi,g)$.

\begin{theorem}\label{l:nonmove}
If Marker has a strategy against restricted Cutter limiting the value of each game state to $t$, then Marker also has such a strategy against combinatorial Cutter.
\end{theorem}

\begin{proof}
To fix some notation, we show that if Marker has a strategy $\Phi$ in the game $\mathcal{G}_{RC}$ against restricted Cutter limiting the value of each game state to $t$ indefinitely, then Marker also has a strategy $\Psi$ in the game $\mathcal{G}_{C}$ against combinatorial Cutter limiting the value of each game state to $t$ indefinitely.

Given $\Phi$, we will recursively define such a strategy $\Psi$ in the game $\mathcal{G}_{C}$ for Marker preserving at all times the property $(\star)$ that for any game state $(D,\chi,g)$ occurring in the game $\mathcal{G}_{C}$ played according to $\Psi$, there is an associated game state $f(D,\chi,g)$ occurring in the game $\mathcal{G}_{RC}$ played according to $\Phi$ such that $(D,\chi,g) \preccurlyeq f(D,\chi,g)$, i.e.\ so that $(D,\chi,g)$ is equivalent to a reduction of $f(D,\chi,g)$. Then we clearly have $v(D,\chi,g) \leq v(f(D,\chi,g)) \leq t$, and hence any such strategy $\Psi$ limits the value of each game state to $t$ indefinitely.

For the recursive construction, suppose we are given game states $(D,\chi,g)$ and $f(D,\chi,g)=(D',\chi',g')$ as above. Let $(D',\chi',g',v',w')$ be the response of Marker to $(D',\chi',g')$ according to $\Phi$. Since $(D,\chi,g) \preccurlyeq (D',\chi',g')$, the marked game state $(D',\chi',g',v',w')$ induces a marked game state $(D,\chi,g,v,w)$, which we will take as the response of Marker to $(D,\chi,g)$ according to $\Psi$. More precisely, if $v'$ lies on a cycle of $D'$ which is not contracted to a single point, let $v$ be the image of $v'$ under the contraction map. Otherwise, let $v$ be a dummy vertex. Do the same for $w$, with the additional condition that $v$ and $w$ are the same dummy vertex if and only if $v'$ and $w'$ lay on the same (contracted) cycle or were the same dummy vertex.

It remains to ensure that property $(\star)$ holds for all possible responses $(D^{(1)},\chi^{(1)},g^{(1)})$ to $(D,\chi,g,v,w)$ of combinatorial Cutter in the game $\mathcal{G}_{C}$. Suppose combinatorial Cutter makes a move of type $(\dagger)$ (where $\dagger = a,b,c$ or $d$). 
It may be that restricted Cutter cannot make a move of type $(\dagger)$ in response to $(D',\chi',v',w')$, since it would result in a game state $(D^{(2)},\chi^{(2)},g^{(2)})$ equivalent to a reduction of a game state $(D^{(3)},\chi^{(3)},g^{(3)})$ occurring earlier in the game $\mathcal{G}_{RC}$ played according to $\Phi$. However, it follows from our choice of $v$ and $w$ that $(D^{(1)},\chi^{(1)},g^{(1)}) \preccurlyeq (D^{(2)},\chi^{(2)},g^{(2)})$, and hence $(D^{(1)},\chi^{(1)},g^{(1)}) \preccurlyeq (D^{(3)},\chi^{(3)},g^{(3)})$, and we define $f(D^{(1)},\chi^{(1)},g^{(1)}) := (D^{(3)},\chi^{(3)},g^{(3)})$. Otherwise, we let $(D^{(2)},\chi^{(2)},g^{(2)})$ be the game state after restricted Cutter makes a move of type $(\dagger)$. Then as before it follows from our choice of $v$ and $w$ that $(D^{(1)},\chi^{(1)},g^{(1)}) \preccurlyeq (D^{(2)},\chi^{(2)},g^{(2)})$ and so we can define $f(D^{(1)},\chi^{(1)},g^{(1)}) := (D^{(2)},\chi^{(2)},g^{(2)})$.
\end{proof}

\subsection{Forcing moves against restricted Cutter}

Hence, given Lemma~\ref{l:gameequiv} and the previous theorem, we only need to show that combinatorial Marker has a strategy against restricted Cutter to limit the value of each game state to at most $t$. This is helpful, as it turns out that in a number of cases described in the following lemmas, Marker can force restricted Cutter to make specific responses.

\begin{defn}
Given a cycle $C$ with edge labelling $\chi$ and labels $a,b$ occurring in $\chi$, we say that vertices $v,w \in V(C)$ \emph{gather} $a$, if the label $a$ only occurs on one segment of $C$ between $v$ and $w$. We say that $v,w \in V(C)$ \emph{separate} $a$ and $b$, if the label $a$ only occurs on one segment, and the label $b$ only occurs on the other segment of $C$ between $v$ and $w$.
\end{defn}

\begin{lemma}\label{l:limitedchoice}
Suppose for some game state $(D_k,\chi_k,g_k)$ there is a directed cycle $C$ with an isolated label $a$. Assume that Marker chooses vertices $v,w \in V(C)$ which gather $a$, and that $a$ occurs as a label on $\hat C_1$. Then the move (\ref{movec}) is not legal for restricted Cutter.
\end{lemma}

\begin{proof}
A move of type (\ref{movec}) is equivalent to a reduction of $(D_k,\chi_k,g_k)$ after contracting all edges of $C$ which appear in $\hat{C}_1$ except one edge labelled $a$.
\end{proof}

\begin{lemma}\label{l:nochoice}
Suppose for some game state $(D_k,\chi_k,g_k)$ there is a directed cycle $C$ with isolated labels $a$ and $b$. If Marker chooses vertices $v,w \in V(C)$ separating the labels $a$ and $b$, then moves of type (\ref{moveb}) and (\ref{movec}) are not legal for restricted Cutter. In particular, he either has to make a move of type (\ref{movea}), or has no legal move.
\end{lemma}

\begin{proof}
Apply Lemma \ref{l:limitedchoice} twice, exchanging the roles of $\hat C_1$ and $\hat C_2$ for the second application, and note that (\ref{moved}) is not a legal move since $v$ and $w$ lie on the same cycle.
\end{proof}

As a simple consequence of the previous two lemmas, we note that if Marker is playing against restricted Cutter then the value of the game state must increase after every turn.

\begin{cor}\label{c:valueincrease}
For any game state $\gamma$ and any move of Marker, if restricted Cutter has a legal move in response then for the resulting gamestate $\gamma'$ we have that $v(\gamma') = v(\gamma) +1$.
\end{cor}
\begin{proof}
Moves of type (\ref{movea}) or (\ref{moved}) always increase the value of the game state by one. Suppose for contradiction that restricted Cutter makes a move of type (\ref{moveb}) or (\ref{movec}) and $v(\gamma') \leq v(\gamma)$. Let $v,w$ be the elements that Marker chose, lying on the cycle $C$. Since $v(\gamma') \leq v(\gamma)$ we may assume that $C$ is non-trivial.

Since there is always a new label in $\gamma'$ which doesn't appear in $\gamma$, the only way that $v(\gamma') \leq v(\gamma)$ is if there is some label $\ell$ in $\gamma$ which doesn't appear in $\gamma'$. However, the label $\ell$ must then be gathered by $v$ and $w$, and so Lemma \ref{l:limitedchoice} implies that any move in which $\ell$ does not appear in $\gamma'$ is not a legal move for restricted Cutter.
\end{proof}

\section{Strategies for the Combinatorial Game}

Our aim in this section is to show that if both players play optimally in the combinatorial Marker-Cutter game, then the limit supremum of the value will be just over $\frac43g_0$. A key concept will be a number called the {\em potential} of the game, which will remain constant under optimal play. The potential is given by the sum of $4(g_0-g) - 3v(\gamma)$ with an extra term which, roughly speaking, encodes the extent to which the structure of the components of $D$ may allow Marker to force Cutter to make genus-reducing moves in the near future.

%Old definition of potential
%More precisely, we define the {\em potential} $p(C, D)$ of a component $C$ of $D$ with say $n$ edges in which there are $l$ pairs of consecutive edges with the same label and $m$ of the edge labels are uniquely appearing in $D$ to be $\frac12(n + l) + m - 2$. Equivalently, every such cycle $C$ starts from potential $-2$, with each edge with a uniquely appearing label contributing an additional value of $3/2$, each pair of successive edges with the same label together also contributing $3/2$, and all remaining edges each contributing $1/2$. Note that no cycle with a single edge can have positive potential, while every cycle with at least five edges does have positive potential.

%The \emph{potential} of a game state $\gamma = (D, \chi, g)$ is then
%\[
%p(\gamma) := 4(g_0 - g) - 3v(D) + \sum_{\substack{C \in \text{Comp}(D) \\ p(C, D) > 0}} p(C, D)
%\]

More precisely, let $\gamma = (D, \chi, g)$ be a game state, let $S$ be a path or cycle in $D$, and let $e$ be an edge contained in $S$. Then we define the \emph{potential}
\[
p(e,S,\gamma) = 
\begin{cases}
\frac 32 & \text{if the label of $e$ is uniquely appearing},\\
\frac 34 & \text{if $e$ is incident to an edge on $S$ with the same label as $e$},\\
\frac 12 & \text{otherwise}.
\end{cases}
\]
Define $p(S,\gamma) = -2+\sum_{e \in S} p(e,S,\gamma)$. The \emph{potential} of the game state $\gamma$ is finally defined by
\[
p(\gamma) := 4(g_0 - g) - 3v(\gamma) + \sum_{\substack{C \in \text{Comp}(D) \\ p(C, \gamma) > 0}} p(C, \gamma)
\]

\subsection{Effect of different moves on the potential}
Since the strategies that we give for Marker and Cutter hinge on bounding $p(\gamma)$, it is beneficial to investigate how different moves in the game effect the potential. For this purpose let $\gamma_k = (D_k, \chi_k, g_k)$ and $\gamma_{k+1} = (D_{k+1}, \chi_{k+1}, g_{k+1})$ be consecutive game states, and let $v,w$ be the elements chosen by Marker. Throughout this section, we will assume that Cutter plays as restricted Cutter,  and so by Corollary \ref{c:valueincrease} the value of the game state will increase by one on each turn.

\begin{lemma}
\label{lem:movedbound}
Assume that $v$ and $w$ lie in different components $C$ and $C'$ of $D_k$ respectively, and consequently restricted Cutter has to respond with a move of type (\ref{moved}).
\begin{enumerate}
    \item $p(\gamma_{k+1}) \leq p(\gamma_k)$.
    \item If $p(C,\gamma_k) \geq 0$, $p(C',\gamma_k) \geq 0$ and both $v$ and $w$ are incident to edges with different labels, then $p(\gamma_{k+1}) = p(\gamma_k)$.
\end{enumerate} 
\end{lemma}

\begin{proof}
Let $\hat C$ be the amalgamated cycle resulting from the move of type (\ref{moved}). For every edge $e$ in $C$ we have that $p(e,\hat C,\gamma_{k+1}) \leq p(e,C,\gamma_k)$ with equality unless $v$ is incident to $e$ and another edge with the same label as $e$. An analogous statement holds for $C'$ and $w$. If neither $C$ nor $C'$ are trivial, then $\hat C$ contains two additional edges with potential $\frac 12$, and we can conclude that 
\[
p(\hat C,\gamma_{k+1}) - p(C,\gamma_k) - p(C',\gamma_k) \leq 3
\]
with equality if and only if both $v$ and $w$ are incident to edges with different labels. If we interpret the potential of a trivial cycle as $-2$, it is easily checked that the above still holds when $C$ or $C'$ are trivial. The potentials of all other cycles are the same in $\gamma_k$ and $\gamma_{k+1}$. 

Finally note that if $p(\hat C,\gamma_{k+1}) < 0$, then the same must hold for $p(C,\gamma_k)$ and $p(C',\gamma_k)$ and that $g_{k+1} =  g_k$ by definition of (\ref{movec}). The lemma follows from plugging these observations into the definition of $p(\gamma)$, recalling that $v(\gamma_{k+1}) = v(\gamma_k) +1$.
\end{proof}

Next consider the case where $v$ and $w$ lie on the same cycle $C$. As in the previous section, define $\hat C_1 := P \cup \{f\}$ and $\hat C_2 := P' \cup \{f\}$.

\begin{lemma}
\label{lem:moveabound}
Assume that restricted Cutter chooses response (\ref{movea}). 
\begin{enumerate}
    \item If both $p(P,\gamma_k) \geq -\frac 12$ and $p(P',\gamma_k) \geq -\frac 12$, then $p(\gamma_{k+1}) \leq p(\gamma_k)$.
    \item If $v$ and $w$ gather each label that appears on $C$, then $p(\gamma_{k+1}) \geq p(\gamma_k)$.
\end{enumerate}
\end{lemma}

\begin{proof}
As in the previous lemma, we start by observing that for every edge $e$ in $P$ we have that $p(e,\hat C_1,\gamma_{k+1}) \leq p(e,C,\gamma_k)$ and similarly for $P'$ and $\hat C_2$. If $v$ and $w$ gather the label of $e$, then equality holds. Consequently,
\[
p(\hat C_1,\gamma_{k+1}) + p(\hat C_2,\gamma_{k+1}) - p(C,\gamma_k) \leq -1
\]
with equality in case $v$ and $w$ gather each label that appears on $C$. Again, if $C$ is trivial we interpret its potential as $-2$ in the above. Potentials of all other cycles are the same in $\gamma_k$ and $\gamma_{k+1}$.

Now, if $p(P,\gamma_k) \geq -\frac 12$ and $p(P',\gamma_k) \geq -\frac 12$, then $p(\hat C_1,\gamma_{k+1}) \geq 0$, $p(\hat C_2,\gamma_{k+1}) \geq 0$, and $p(C,\gamma_k) \geq 0$. Recall that $g_{k+1} = g_k - 1$ by definition of (\ref{movea}) and so $(1)$ follows from the definition of $p(\gamma)$. Assuming that $v$ and $w$ gather each label that appears on $C$, to show $(2)$ we need to show that
\begin{equation}\label{e:max}
\max \{ 0, p(\hat C_1,\gamma_{k+1}) \} + \max \{0, p(\hat C_2,\gamma_{k+1})\} - \max \{ 0, p(C,\gamma_k)\} \geq  -1
\end{equation}
Since $p(\hat C_1,\gamma_{k+1}) + p(\hat C_2,\gamma_{k+1}) - p(C,\gamma_k) \leq -1$, if $p(C,\gamma_k) \geq 0$ then (\ref{e:max}) holds. However, if $p(C,\gamma_k)<0$ then the length of $C$ is $\leq 3$, and in each case it is a simple check that (\ref{e:max}) holds, again recalling that $v(\gamma_{k+1}) = v(\gamma_k) +1$.
\end{proof}

Lemmas \ref{lem:movedbound} and \ref{lem:moveabound} give conditions for Cutter to ensure that the potential does not increase and for Marker to make sure that the potential does not decrease, if moves of type (\ref{moved}) or (\ref{movea}) are played respectively. In case Cutter can chose a response of type (\ref{moveb}) or (\ref{movec}), it is slightly more difficult for Marker to make sure that the potential does not increase. To this end, we introduce the following notion.

\begin{defn}
\label{defn:nesting}
A \emph{nesting path} is one of the following:
\begin{itemize}
    \item a single edge with a uniquely appearing label,
    \item a path with non-isolated labels $(x,y,z)$ such that 
    \begin{itemize}
        \item $x$ and $z$ also occur on a cycle with labels $(x,t,z,t)$ for some label $t$, and
        \item $y$ also occurs on a cycle consisting of an edge with label $y$ and a nesting path.
    \end{itemize} 
\end{itemize}
\end{defn}

Note, since the game state is proper, a nesting path always contributes $\frac 32$ to the potential of the cycle it lies in. 
%Further note that if Cutter makes a move of type (\ref{moveb}) and there is a nesting path $S$ contained in $\hat C_2$, then $S$ cannot be a single edge due to Lemma \ref{l:limitedchoice}. The labels $x,y,z$ which occur on $S$ become uniquely appearing and thus we get two new active cycles, one with labels $(x,t,z,t)$, and the other consisting of an edge with label $y$ and a nesting path. An analogous observation can be made for moves of type (\ref{movec}).

\begin{lemma}
\label{lem:movebcbound}
Assume that restricted Cutter chooses response (\ref{moveb}).
\begin{enumerate}
    \item If $p(P',\gamma_k) < -\frac 12$, then $p(\gamma_{k+1}) \leq p(\gamma_k)$. 
    \item If $P'$ is a nesting path, then $p(\gamma_{k+1}) = p(\gamma_k)$. 
\end{enumerate}

Assume that restricted Cutter chooses response (\ref{movec}).
\begin{enumerate}[resume]
    \item If $p(P,\gamma_k) < -\frac 12$, then $p(\gamma_{k+1}) \leq p(\gamma_k)$. 
    \item If $P$ is a nesting path, then $p(\gamma_{k+1}) = p(\gamma_k)$. 
\end{enumerate}
\end{lemma}

\begin{proof}
By symmetry, it suffices to prove the statement for moves of type (\ref{moveb}).
Let us first prove the upper bound. Since $p(P',\gamma_k) < -\frac 12$, the path $P'$ consists of $\ell \leq 2$ edges none of which has a uniquely appearing label. In particular, $\ell$ edges have potential one larger in $\gamma_{k+1}$ than they have in $\gamma_k$. The label of $f$ is also uniquely appearing in $D_{k+1}$ and the new cycle $\hat{C}_1$ has $\ell - 1$ fewer edges than before. Summing up these different contributions we arrive at 
\[
\sum_{\substack{C \in \text{Comp}(D_{k+1})\\p(C, \gamma_{k+1}) > 0}} p(C, \gamma_{k+1}) \leq \sum_{\substack{C \in \text{Comp}(D_k)\\ p(C, \gamma_k) > 0}} p(C, \gamma_k) + 3.
\]
Since $g_{k+1} = g_k$ this concludes the proof of the first part. 

If $P'$ is a nesting path, then we need to show that equality holds in the above equation. First note that $P'$ cannot be a single edge due to Lemma \ref{l:limitedchoice} and our assumption that Cutter plays as restricted Cutter. In particular, $P'$ has labels $(x,y,z)$ and there will be a cycle $(x,t,z,t)$ whose potential increases from $0$ to $2$ (as now $x$ and $z$ are uniquely appearing) and a cycle consisting of an edge with label $y$ and a nesting path whose potential increases from $0$ to $1$ (as now $y$ is uniquely appearing). Since the label of $f$ is uniquely appearing, it contributes $\frac 32$ to the potential of $\hat C_1$. We  thus have $p(\hat C_1,\gamma_{k+1}) = p(C,\gamma_k)$, and since the potential of all other cycles remains unchanged, we have indeed equality in the above equation, concluding the proof.
\end{proof}

\subsection{Strategies for the Marker-Cutter game}

The results in this section imply Theorem \ref{t:PC}. More precisely, we show that restricted Cutter has a strategy ensuring that the game reaches a state of value at least $\frac43g_0 + 2$, whereas Marker has a strategy against restricted Cutter to make sure that the value of each game state is at most $\frac43g_0 + 4$.

Both of these strategies work by bounding the potential throughout the game. More precisely, restricted Cutter's best strategy is to make sure that the potential never becomes too large.

\begin{theorem}\label{thm:comblowerbound}
Restricted Cutter has a strategy in the combinatorial Marker-Cutter game to ensure that the game reaches a state of value at least $\frac43g_0 + 2$.
\end{theorem}

\begin{proof}
By Lemmas \ref{lem:movedbound}(1), \ref{lem:moveabound}(1), and \ref{lem:movebcbound}(1)\&(3), restricted Cutter has a strategy to ensure that the potential never increases, as long as the moves suggested by these lemmas are always legal for restricted Cutter. However, we note that for each of the moves suggested by this strategy the value of the game state will increase by one each turn. Indeed, this is clear for moves of type (\ref{movea}) and (\ref{moved}) and if an isolated label were to appear only on a path $P$ then $p(P,\gamma) \geq -\frac{1}{2}$ and so the strategy will not call to make a move of type (\ref{movec}) to remove this label, and similarly for moves of type (\ref{moveb}). In particular, none of these moves which are legal for Cutter can be illegal for restricted Cutter since the value of the resulting game state $\gamma_{k+1}$ is larger than the value of any earlier game state $\gamma_i$ with $i \leq k$, and hence $\gamma_k$ cannot be equivalent to a reduction of $\gamma_i$ with $i \leq k$.

It is a simple check that the first two moves of Cutter according to this strategy will be of types (\ref{moveb}), (\ref{movec}), or (\ref{moved}), and so $g_2 = g_0$. In $\chi_2$, there are two different labels, and it is not hard to show that the combined potentials of all components with non-negative potential sum up to at most $1$. Hence we have that $p(\gamma_2) \leq -5$.

If the moves suggested by restricted Cutter's strategy always remain legal, then $v(\gamma_k)$ will become arbitrarily large, and in particular larger than $\frac 43 g_0 + 2$. Therefore, we may assume that there is a $k$ such that the suggested move is not a legal move for restricted Cutter, and hence not legal for Cutter (i.e.\ a move of type (\ref{movea}) with $g_k = 0$). Since the potential never decreased we have that 
\[
-5 \geq p(\gamma_k) = 4g_0 - 3v(\gamma_k) + \sum_{\substack{C \in \text{Comp}(D) \\ p(C, \gamma_k) > 0}} p(C, \gamma_k)
\]
However, since restricted Cutter's strategy calls for a move of type (\ref{movea}) there is a cycle $C$ composed of two paths $P$ and $P'$ with $p(P,\gamma_k),p(P',\gamma_k) \geq -\frac{1}{2}$ and hence $p(C,\gamma_k) = p(P,\gamma_k) + p(P',D_k) + 2 \geq 1$. Hence $-5 \geq p(\gamma_k) = 4g_0 - 3v(\gamma_k) + 1$, giving $v(\gamma_k) \geq \frac43g_0 + 2$.
\end{proof}

Similarly, Marker's best strategy in the restricted Marker-Cutter game is to ensure that from some point on the potential no longer decreases. Lemmas~\ref{lem:movedbound}(2), \ref{lem:moveabound}(2), and \ref{lem:movebcbound}(2)\&(4) provide Marker with conditions to do so.

However, this alone will not suffice; it could be that the potential and the genus both remain constant but $v(\gamma)$ becomes arbitrarily large, provided this is balanced by a corresponding increase in the term 
\[\sum_{\substack{C \in \text{Comp}(D)\\ p(C, \gamma) > 0}} p(C, \gamma)\,.\]

To avoid this issue, Marker will also maintain some control over the collection of components of positive potential. More precisely, from some point on Marker will ensure that the current set of boundary cycles always consists of a set of {\em passive} cycles of non-positive potential, together with a set of {\em active} cycles, given (up to equivalence, cf.\ Definition~\ref{def_equivalence}) by one of the finitely many options depicted in Figure~\ref{fig_Markerstrategy}. Note that active cycles may have positive or negative potential.
%(Josh) : Or zero....is the above bad? It seemed horrible to write positive, negative or zero potential.
Marker will always choose points on the active cycles.

It is, however, possible that previously passive cycles become active again, if Cutter makes moves of type (\ref{moveb}) or (\ref{movec}). More precisely, some labels on passive cycles could become uniquely appearing, making the potentials of these cycles positive. 
%Although he cannot completely avoid this phenomenon, Marker can keep it under careful control by regulating the structure of the boundary cycles - not just of the active ones, but also the passive ones, since they may become active again.
Although he cannot completely avoid this phenomenon, Marker can keep it under control by making sure that whenever it occurs it is caused by a nesting path.

More precisely, if the path $P'$ in a move of type (\ref{moveb}) is a nesting path, then we know by Lemma \ref{l:limitedchoice} that it cannot consist of a single edge. Hence $P'$ is a path of length $3$ with labels $(x,y,z)$, and as mentioned in the proof of Lemma \ref{lem:movebcbound}, exactly two previously passive cycles become active, one labelled by $(x,t,z,t)$ and the other consisting of an edge with label $y$ and a nesting path (where $x$, $y$, and $z$ have become uniquely appearing). A similar observation can be made for moves of type (\ref{movec}).

\usetikzlibrary{hobby}
\usetikzlibrary{decorations.pathmorphing}
\usetikzlibrary{decorations.markings}
\usetikzlibrary{pgfplots.fillbetween}
\usetikzlibrary{decorations.pathmorphing, arrows, calc, shapes, matrix, arrows.meta}

\definecolor{myred}{rgb}{.9,.1,.2}

 \newcommand\SqSmall[6]{%
    \draw[-Latex] (#1, #2) -- (#1, #2+1) ;
     \draw[-Latex] (#1, #2+1) --(#1+1, #2+1);
      \draw[-Latex] (#1+1, #2+1) --(#1+1, #2) ;
       \draw[-Latex](#1+1, #2)--(#1, #2) ;
    \node at (#1-.25, #2+.5) {$#3$}; 
    \node at (#1+.5, #2+1.25) {$#4$}; 
    \node at (#1+1.25, #2+.5) {$#5$}; 
    \node at (#1+.5, #2-.25) {$#6$}; 
    \filldraw (#1, #2) circle (1pt);
     \filldraw (#1, #2+1) circle (1pt);
      \filldraw (#1+1, #2+1) circle (1pt);
       \filldraw (#1+1, #2) circle (1pt) ;
}
 \newcommand\SquigglySq[6]{%
    \draw[-Latex,decorate] (#1, #2) -- (#1, #2+1) ;
     \draw[-Latex] (#1, #2+1) --(#1+1, #2+1);
      \draw[-Latex] (#1+1, #2+1) --(#1+1, #2) ;
       \draw[-Latex](#1+1, #2)--(#1, #2) ;
    \node at (#1-.25, #2+.5) {$#3$}; 
    \node at (#1+.5, #2+1.25) {$#4$}; 
    \node at (#1+1.25, #2+.5) {$#5$}; 
    \node at (#1+.5, #2-.25) {$#6$}; 
    \filldraw (#1, #2) circle (1pt);
     \filldraw (#1, #2+1) circle (1pt);
      \filldraw (#1+1, #2+1) circle (1pt);
       \filldraw (#1+1, #2) circle (1pt) ;
}
\tikzstyle{densely dashed}=          [dash pattern=on 3pt off 2pt]

\tikzset{
number1/.pic={
\begin{scope}[shift={(-.5,0)}]
\filldraw (0, 0) circle (1pt) ;
\filldraw (1, 0) circle (1pt) ;
\draw[-Latex] (0,0) to [bend right] (1,0);
\draw[-Latex, decorate] (1,0) to [bend right] (0,0);

\draw[densely dashed] (0,0) to (1,0);
\end{scope}
}
}

\tikzset{
number2/.pic={

\begin{scope}[shift={(-.5,-.75)}]
\filldraw (0,0) circle (1pt) ;
\filldraw (1,0) circle (1pt) ;
\draw[-Latex] (0,0) to [bend right] (1,0);
\draw[-Latex] (1,0) to [bend right] (0,0);
\node at (.5, .35) {$a$}; 

\filldraw (0, 1.5) circle (1pt) ;
\filldraw (1, 1.5) circle (1pt) ;
\draw[-Latex] (0,1.5) to [bend right] (1,1.5);
\draw[-Latex, decorate] (1,1.5) to [bend right] (0,1.5);
\node at (.5, 1.15) {$a$}; 

\draw[densely dashed] (1,0) to [bend right] (1,1.5);
\end{scope}
}
}

\tikzset{
number3/.pic={
\begin{scope}[shift={(-.5,-1)}]
\draw[-Latex,decorate] (0,0) -- (0,1) ;
     \draw[-Latex] (0,1) --(0,2);
      \draw[-Latex] (0,2) --(1,2);
       \draw [-Latex] (1,2)--(1,1) ;
    \draw [-Latex] (1,1)--(1,0) ;
    \draw [-Latex] (1,0)--(0,0) ;
     \node at (-.25, 1.5) {$a$}; 
    \node at (.5, 2.25) {$b$}; 
    \node at (1.25, 1.5) {$a$}; 
    \node at (.5,-.25) {$b$}; 
    \filldraw (0,0) circle (1pt);
     \filldraw (0,1) circle (1pt);
      \filldraw (0,2) circle (1pt);
       \filldraw (1,2) circle (1pt) ;
       \filldraw (1,1) circle (1pt) ;
       \filldraw (1,0) circle (1pt) ;

\draw[densely dashed] (0,1) to (1,1);
\end{scope}
}
}

\tikzset{
number4/.pic={
\begin{scope}[shift={(-.5,-1.5)}]
\draw[densely dashed] (0,0) to [bend right=60] (0,1);

\SqSmall{0}{2}{a}{b}{a}{c}

\SquigglySq{0}{0}{}{c}{}{b}
\end{scope}
}
}

\tikzset{
number5/.pic={
\begin{scope}[shift={(-1.5,-1.5)}]
\filldraw (0, 2.5) circle (1pt) ;
\filldraw (1, 2.5) circle (1pt) ;
\draw[-Latex] (0,2.5) to [bend right] (1,2.5);
\draw[-Latex, decorate] (1,2.5) to [bend right] (0,2.5);

\draw[densely dashed] (0,1) to [bend right=60] (1,1);

\SqSmall{0}{0}{d}{}{d}{}

\SqSmall{2}{2}{a}{b}{a}{c}

\SqSmall{2}{0}{}{c}{}{b}
\end{scope}
}
}

\tikzset{
number6/.pic={
\begin{scope}[shift={(-1.5,-1.5)}]
\filldraw (0,2) circle (1pt) ;
\filldraw (1, 2) circle (1pt) ;
\draw[-Latex] (0,2) to [bend right] (1,2);
\draw[-Latex] (1,2) to [bend right] (0,2);
\node at (.5, 1.65) {$e$}; 

\filldraw (0, 3) circle (1pt) ;
\filldraw (1, 3) circle (1pt) ;
\draw[-Latex] (0,3) to [bend right] (1,3);
\draw[-Latex, decorate] (1,3) to [bend right] (0,3);

\draw[densely dashed] (0,1) to [bend right=60] (1,1);

\SqSmall{0}{0}{d}{}{d}{e}

\SqSmall{2}{2}{a}{b}{a}{c}

\SqSmall{2}{0}{}{c}{}{b}
\end{scope}
}
}

\tikzset{
number7/.pic={
\begin{scope}[shift={(-1.5,-1.5)}]

\filldraw (0, 1.5) circle (1pt) ;
\filldraw (1, 1.5) circle (1pt) ;
\draw[-Latex] (0,1.5) to [bend right] (1,1.5);
\draw[-Latex, decorate] (1,1.5) to [bend right] (0,1.5);

\draw[densely dashed] (2,1) to [bend left=60] (2,0);

\SqSmall{2}{2}{a}{b}{a}{c}

\SqSmall{2}{0}{}{c}{}{b}
\end{scope}
}
}

\tikzset{
number8/.pic={

\begin{scope}[shift={(-1.5,-1.5)}]
\filldraw (0,1) circle (1pt) ;
\filldraw (1, 1) circle (1pt) ;
\draw[-Latex] (0,1) to [bend right] (1,1);
\draw[-Latex] (1,1) to [bend right] (0,1);
\node at (.5, .5) {$d$}; 

\filldraw (0, 2) circle (1pt) ;
\filldraw (1, 2) circle (1pt) ;
\draw[-Latex] (0,2) to [bend right] (1,2);
\draw[-Latex, decorate] (1,2) to [bend right] (0,2);

\draw[densely dashed] (3,1) to [bend right=60] (3,0);

\SqSmall{2}{2}{a}{b}{a}{c}

\SqSmall{2}{0}{d}{c}{}{b}
\end{scope}
}
}

\tikzset{
number9/.pic={
\begin{scope}[shift={(-1.5,-.5)}]
\filldraw (0,-.5) circle (1pt) ;
\filldraw (1, -.5) circle (1pt) ;
\draw[-Latex] (0,-.5) to [bend right] (1,-.5);
\draw[-Latex] (1,-.5) to [bend right] (0,-.5);

\filldraw (0,.5) circle (1pt) ;
\filldraw (1, .5) circle (1pt) ;
\draw[-Latex] (0,.5) to [bend right] (1,.5);
\draw[-Latex] (1,.5) to [bend right] (0,.5);

\filldraw (0, 1.5) circle (1pt) ;
\filldraw (1, 1.5) circle (1pt) ;
\draw[-Latex] (0,1.5) to [bend right] (1,1.5);
\draw[-Latex, decorate] (1,1.5) to [bend right] (0,1.5);

\draw[densely dashed] (0,-.5) to (1,-.5);

\SqSmall{2}{0}{a}{}{a}{}
\end{scope}
}
}

\tikzset{
number10/.pic={
\begin{scope}[shift={(-1.5,-.5)}]
\filldraw (0,0) circle (1pt) ;
\filldraw (1, 0) circle (1pt) ;
\draw[-Latex] (0,0) to [bend right] (1,0);
\draw[-Latex] (1,0) to [bend right] (0,0);

\filldraw (0, 1) circle (1pt) ;
\filldraw (1, 1) circle (1pt) ;
\draw[-Latex] (0,1) to [bend right] (1,1);
\draw[-Latex, decorate] (1,1) to [bend right] (0,1);

\draw[densely dashed] (2,1) to [bend right=60] (3,1);

\SqSmall{2}{0}{a}{}{a}{}
\end{scope}
}
}

\tikzset{
number11/.pic={
\begin{scope}[shift={(-1.5,-1.5)}]
\filldraw (0,1) circle (1pt) ;
\filldraw (1, 1) circle (1pt) ;
\draw[-Latex] (0,1) to [bend right] (1,1);
\draw[-Latex] (1,1) to [bend right] (0,1);

\filldraw (0, 2) circle (1pt) ;
\filldraw (1, 2) circle (1pt) ;
\draw[-Latex] (0,2) to [bend right] (1,2);
\draw[-Latex, decorate] (1,2) to [bend right] (0,2);

\filldraw (2,2.5) circle (1pt) ;
\filldraw (3, 2.5) circle (1pt) ;
\draw[-Latex] (2,2.5) to [bend right] (3,2.5);
\draw[-Latex] (3,2.5) to [bend right] (2,2.5);
\node at (2.5, 2) {$b$};

\draw[densely dashed] (2,0) to [bend left=60] (3,0);

\SqSmall{2}{0}{a}{b}{a}{}
\end{scope}
}
}

\tikzset{
number12/.pic={

\begin{scope}[shift={(-.5,-.5)}]
\filldraw (0,0) circle (1pt) ;
\filldraw (1, 0) circle (1pt) ;
\draw[-Latex] (0,0) to [bend right] (1,0);
\draw[-Latex] (1,0) to [bend right] (0,0);

\filldraw (0, 1) circle (1pt) ;
\filldraw (1, 1) circle (1pt) ;
\draw[-Latex] (0,1) to [bend right] (1,1);
\draw[-Latex, decorate] (1,1) to [bend right] (0,1);

\draw[densely dashed] (0,0) to (1,0);
\end{scope}
}
}

%\newpage

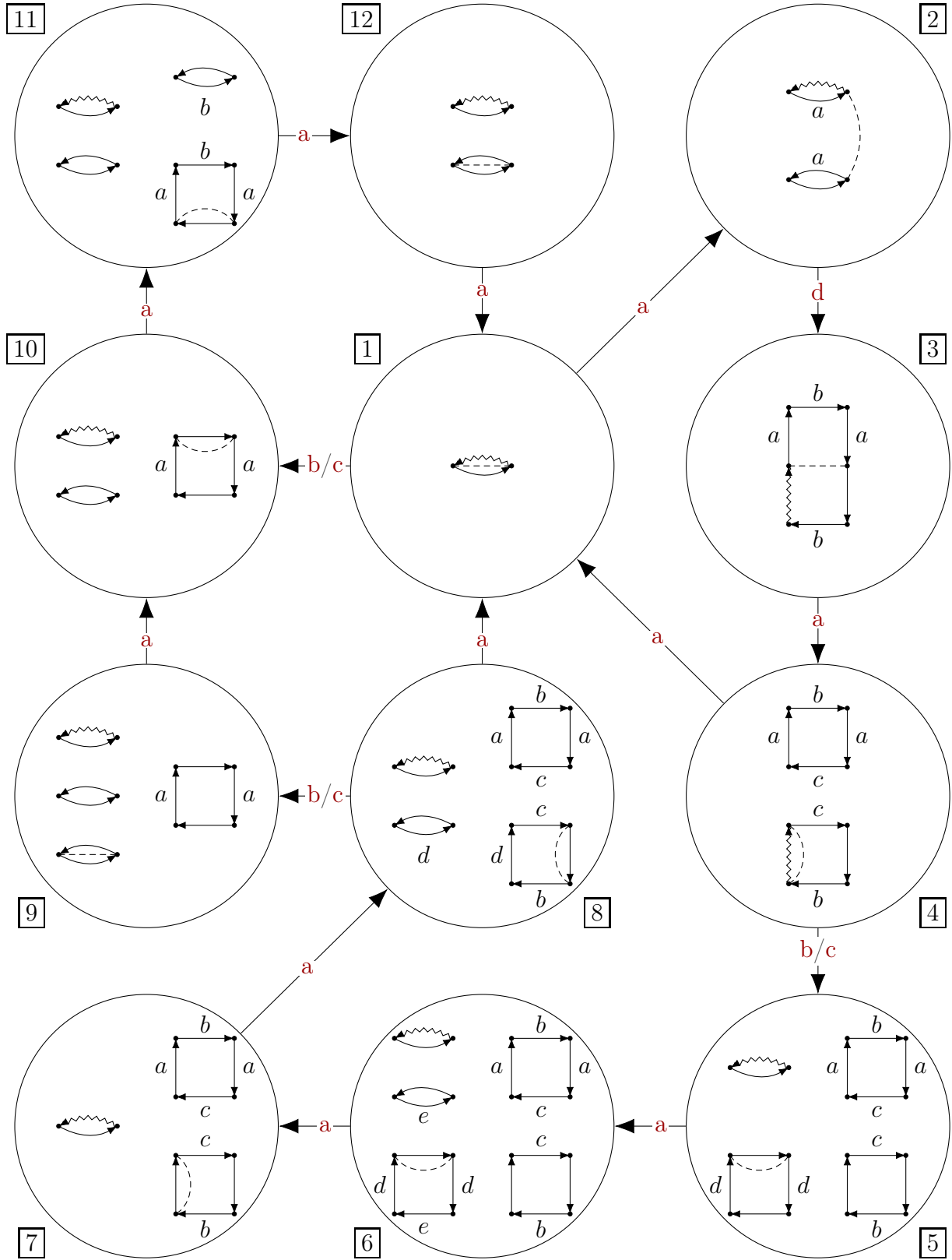
\begin{figure}
\begin{tikzpicture}[scale=0.94, decoration = {zigzag,segment length=4pt,amplitude=1pt},state/.style={draw,circle,minimum size=4.5cm}]

\node[state,label=above left:{\fbox{1}}] (s1) at (0,3){};
\node[state,label =above right:{\fbox{2}}] (s2) at (6.1,9){};
\node[state,label = above right:{\fbox{3}}] (s3) at (6.1,3){};
\node[state,label =below right:{\fbox{4}}] (s4) at (6.1,-3){};
\node[state,label =below right:{\fbox{5}}] (s5) at (6.1,-9){};
\node[state,label =below left:{\fbox{6}}] (s6) at (0,-9){};
\node[state,label =below left:{\fbox{7}}] (s7) at (-6.1,-9){};
\node[state,label =below right:{\fbox{8}}] (s8) at (0,-3){};
\node[state,label =below left:{\fbox{9}}] (s9) at (-6.1,-3){};
\node[state,label =above left:{\fbox{10}}] (s10) at (-6.1,3){};
\node[state,label =above left:{\fbox{11}}] (s11) at (-6.1,9){};
\node[state,label =above left:{\fbox{12}}] (s12) at (0,9){};

\draw[-{Latex[scale=2]}] (s1) to node[pos=.45,inner sep=1pt,fill=white]{\ref{movea}} (s2);
\draw[-{Latex[scale=2]}] (s2) to node[pos=.35,inner sep=1pt,fill=white]{\ref{moved}} (s3);
\draw[-{Latex[scale=2]}] (s3) to node[pos=.35,inner sep=1pt,fill=white]{\ref{movea}} (s4);
\draw[-{Latex[scale=2]}] (s4) to node[pos=.35,inner sep=1pt,fill=white]{\ref{moveb}/\ref{movec}} (s5);
\draw[-{Latex[scale=2]}] (s5) to  node[pos=.35,inner sep=1pt,fill=white]{\ref{movea}} (s6);
\draw[-{Latex[scale=2]}] (s6) to  node[pos=.35,inner sep=1pt,fill=white]{\ref{movea}} (s7);
\draw[-{Latex[scale=2]}] (s7) to  node[pos=.45,inner sep=1pt,fill=white]{\ref{movea}} (s8);
\draw[-{Latex[scale=2]}] (s8) to node[pos=.35,inner sep=1pt,fill=white]{\ref{moveb}/\ref{movec}} (s9);
\draw[-{Latex[scale=2]}] (s9) to node[pos=.35,inner sep=1pt,fill=white]{\ref{movea}} (s10);
\draw[-{Latex[scale=2]}] (s10) to  node[pos=.35,inner sep=1pt,fill=white]{\ref{movea}} (s11);
\draw[-{Latex[scale=2]}] (s11) to  node[pos=.35,inner sep=1pt,fill=white]{\ref{movea}} (s12);

\draw[-{Latex[scale=2]}] (s12) to node[pos=.35,inner sep=1pt,fill=white]{\ref{movea}} (s1);
\draw[-{Latex[scale=2]}] (s4) to node[pos=.45,inner sep=1pt,fill=white]{\ref{movea}} (s1);
\draw[-{Latex[scale=2]}] (s8) to node[pos=.35,inner sep=1pt,fill=white]{\ref{movea}} (s1);

\draw[-{Latex[scale=2]}] (s1) to node[pos=.35,inner sep=1pt,fill=white]{\ref{moveb}/\ref{movec}} (s10);

\pic[shift={(s1)}]{number1};
\pic[shift={(s2)}]{number2};
\pic[shift={(s3)}]{number3};
\pic[shift={(s4)}]{number4};
\pic[shift={(s5)}]{number5};
\pic[shift={(s6)}]{number6};
\pic[shift={(s7)}]{number7};
\pic[shift={(s8)}]{number8};
\pic[shift={(s9)}]{number9};
\pic[shift={(s10)}]{number10};
\pic[shift={(s11)}]{number11};
\pic[shift={(s12)}]{number12};
\end{tikzpicture}

\caption{Possible configurations of active cycles in Marker's strategy, see Table \ref{tab:legend} for an explanation of the diagrams. Arrows from one configuration to another represent restricted Cutter's possible replies from (\ref{movea})--(\ref{moved}).}
\label{fig_Markerstrategy}
\end{figure}
\begin{table}[b]
    \centering
    \setlength{\tabcolsep}{1em} % for the horizontal padding
    \renewcommand{\arraystretch}{1.3}% for the vertical padding
    \begin{tabular}{|c|l|}
        \hline
        \begin{tikzpicture}
        \draw[-Latex] (0,0) to  (1,0);
        \end{tikzpicture} 
        & edge with a uniquely appearing label  \\
        \hline
        \begin{tikzpicture}
        \draw[-Latex] (0,0) to  (1,0);
        \node at (.5, .2) {$x$}; 
        \end{tikzpicture} 
        & edge with label $x$ which is not uniquely appearing  \\
        \hline
        \begin{tikzpicture}[decoration = {zigzag,segment length=4pt,amplitude=1pt}]
        \draw[-Latex,decorate] (0,0) to  (1,0);
        \end{tikzpicture} 
        & nesting path, cf.\ Definition \ref{defn:nesting} \\
        \hline
        \begin{tikzpicture}
        \filldraw (0,0) circle (1pt);
        \filldraw (1,0) circle (1pt);
        \draw[densely dashed] (0,0) to  (1,0);
        \end{tikzpicture} 
        & endpoints of dashed lines are the vertices picked by Marker  \\
        \hline
    \end{tabular}
    \caption{Different types of lines and arrows in the diagrams in Figure \ref{fig_Markerstrategy}.}
    \label{tab:legend}
\end{table}

\begin{theorem}\label{thm:comupperbound}
Marker has a strategy in the restricted Marker-Cutter game to ensure that the value of each game state is at most $\frac43g_0 + \frac{10}3$.
\end{theorem}

\begin{proof}
Marker's strategy consists of two phases. A \emph{preparatory} phase during which all cycles have length $\leq 2$ and non-positive potential, followed by the \emph{potential bounding} phase during which the set of active cycles is (up to equivalence) one of the possibilities shown in Figure \ref{fig_Markerstrategy}.

Marker's strategy during the preparatory phase is as follows. If there is no uniquely appearing label, then choose $v=w=v_\gamma$ a dummy vertex. If there is one, then let $v$ and $w$ be the end points of an edge with a uniquely appearing label. In either case, $v$ and $w$ lie on the same cycle, so Cutter can choose to respond with either a move of type (\ref{movea}), or  a move of types (\ref{moveb}) or (\ref{movec}). The preparatory phase ends after the second time Cutter does not chose type (\ref{movea}).

Since every move of type (\ref{movea}) reduces $g$ by one, Cutter cannot choose (\ref{movea}) arbitrarily often. In particular, the preparatory phase must end after finitely many moves. It is readily verified that before the first move of type (\ref{moveb}) or (\ref{movec}), all cycles have length $1$ and no labels are uniquely appearing. For the subsequent moves before the second occurrence of a move of type (\ref{moveb}) or (\ref{movec}), there is exactly one uniquely appearing label, and all cycles have length at most $2$. By Lemma \ref{l:limitedchoice}, the second move of type (\ref{moveb}) or (\ref{movec}) generates a cycle of length $2$ both of whose labels are uniquely appearing, while all other cycles still have non-positive potential. Choosing this cycle as the unique active cycle will put us in situation \fbox{1} of Figure \ref{fig_Markerstrategy} (where one of the edges with uniquely appearing labels is interpreted as a nesting path).

To bound the potential at the end of the preparatory phase, note that before the first move the potential is $0$. By Lemma \ref{lem:moveabound}, none of the moves of type (\ref{movea}) decreases the potential. The first move of type (\ref{moveb}) or (\ref{movec}) increases $v(D)$ by one without changing $g$ or creating any cycles of positive potential, thus reducing the potential by $3$. The second move of type (\ref{moveb}) or (\ref{movec}) creates one cycle of potential $1$, thus reducing the potential by $2$. Hence at the end of the preparatory phase we have $p(\gamma_k) \geq -5$.

Once the preparatory phase has ended, Marker enters the potential bounding phase. As explained above, at any point during this phase the set of boundary cycles will consist of a set of passive cycles and a set of active cycles where only active cycles are allowed to have positive potential and the set of active cycles will be one of the possibilities indicated in Figure~\ref{fig_Markerstrategy}, see Table \ref{tab:legend} for an explanation of the diagrams. Marker's move in each case is to choose the two endpoints of the dotted line, and the arrows between different game states correspond to Cutter's replies from (\ref{movea})--(\ref{moved}). In most cases, the transitions are easy to check, so we will only outline the arguments and leave the details to the reader.

If we are in \fbox{1}, then the two vertices picked by Marker lie on the same cycle, so Cutter can respond with a move of type (\ref{movea}), or with one of types (\ref{moveb}) or (\ref{movec}). In the first case, all passive cycles retain their current non-positive potential, and in particular we are not forced to make any of them active. The unique active cycle $C$ is replaced by two new cycles $\hat C_1$ and $\hat C_2$, and choosing both of them to be active we end up at \fbox{2}. In the second case, we know by Lemma \ref{l:nochoice} that the nesting path cannot consist of a single edge. By Lemma \ref{l:limitedchoice}, we either get a move of type (\ref{moveb}) with a nesting path $P'$, or a move of type (\ref{movec}) with a nesting path $P$. Consequently, there are two cycles whose potential becomes positive: one with labels $(x,a,z,a)$ where $x$ and $z$ are uniquely appearing, and one consisting of an edge with a uniquely appearing label $y$ and a nesting path. Since the label of the new edge $f$ is also uniquely appearing, this puts us into situation \fbox{10}.

In \fbox{2}, the points chosen by Marker lie in different cycles $C$ and $C'$, so Cutter must respond with a move of type (\ref{moved}). The potential of all passive cycles remains unchanged, so we are not forced to make any of them active and choosing the amalgamated cycle $\hat C$ as the unique active cycle puts us into situation \fbox{3}.

In \fbox{3}, the vertices chosen by Marker separate label $a$ from a uniquely appearing label. By Lemma \ref{l:nochoice}, Cutter has to respond with a move of type (\ref{movea}). The potential of all passive cycles remains unchanged. The unique active cycle $C$ is replaced by two new cycles $\hat C_1$ and $\hat C_2$, and choosing both of them to be active we end up at \fbox{4}. Analogous arguments can be used to treat situations \fbox{5}, \fbox{6}, \fbox{7}, \fbox{9}, \fbox{10}, \fbox{11}, and \fbox{12}, in some of these cases we additionally choose resulting cycles of potential $0$ to become passive.
%(Josh) : Specifically they're always 4-cycles with all edges non-isolate and distinct, no? I don't know if it helps to mention this.

In \fbox{4}, Cutter once again can respond with a move of type (\ref{movea}), or with one of type (\ref{moveb}) or (\ref{movec}). Let $l$ be the label of the new edge $f$. If Cutter chooses a response of type (\ref{movea}), then the path with labels $(b,l,c)$ satisfies the conditions for a nesting path, and if we choose all cycles of potential $0$ to become passive we end up at \fbox{1}. Otherwise, by arguments similar to those from \fbox{1}, we get two new active cycles, one with labels $(x,d,z,d)$ where $x$ and $z$ are uniquely appearing, and one consisting of an edge with a uniquely appearing label and a nesting path. Since $l$ in this case is also uniquely appearing, we arrive at \fbox{5}. 

Finally, in \fbox{8}, if Cutter chooses a response of type (\ref{movea}), then marking all cycles of potential $0$ passive afterwards will put us in situation \fbox{1}. If he chooses a response of type (\ref{moveb}) or (\ref{movec}), then labels $b$,$c$, and $d$  must become uniquely appearing by Lemma \ref{l:limitedchoice}, thus we arrive at \fbox{9}.

Now that we successfully defined a strategy for Marker, it only remains to show that by following this strategy the value of each game state remains bounded. Recall that after the preparatory phase the potential was at least $-5$. Marker's strategy together with Lemmas \ref{lem:movedbound}(2), \ref{lem:moveabound}(2), and \ref{lem:movebcbound}(2)\&(4) ensures that during the potential bounding phase the potential does not decrease. The largest possible value for the term
\[
\sum_{\substack{C \in \text{Comp}(D)\\p(C, \gamma) > 0}} p(C, \gamma)
\]
arises in \fbox{5} and \fbox{9}, where it is equal to 5. So during this phase at any stage we have $-5 \leq 4(g_0 - g) - 3v(\gamma) + 5$, and consequently $v(\gamma) \leq \frac43 g_0 + \frac{10}3$. 
\end{proof}

\section{Bounds for the cop number}

In this section, we use the Marker-Cutter game to derive bounds for the cop number of graphs of a given genus. The following bound is a straightforward consequence of results in the previous sections.

\begin{theorem}
$c(G) \leq \frac{4}{3}g(G) +5$.
\end{theorem}

\begin{proof}
Given a graph $G$ of genus $g$, there is a painting $\mathcal{P}$ of $G$ on an orientable surface $S$ of genus $g$, and hence the surface $F_{\mathcal{P}}$ has genus at most $g$. Since $G \cong G_{\mathcal{P}}$, Theorem \ref{t:strat} implies that $c(G) = c(G_{\mathcal{P}}) \leq v(F_{\mathcal P})+1$. Theorem \ref{t:PC} implies that $v(F_{\mathcal P}) \leq  \frac{4}{3}g  + 4$ and hence $c(G) \leq \frac{4}{3}g +5$.
\end{proof}

In the remainder of this section, we describe some refinements to the cops' strategy which give minor improvements on the above bound. More precisely, we get $c(G) \leq \frac{4}{3}g(G) +3$ if $g(G) \not\equiv 2 \mod 3$, and $c(G) \leq \frac{4}{3}g(G) +4$ if $g(G) \equiv 2 \mod 3$. This easily follows from the following theorem.

\main*

The key idea behind these improvements is that cutlines in the Marker-Cutter game can be seen as guardable walks in the cops and robber game and vice versa. More precisely, recall that the cutline $f_k$ chosen in stage $k$ of the Marker-Cutter game induces a walk in $G_{\mathcal P_k}$, and thus (by applying $\psi_k$) gives a walk in $G = G_{\mathcal P}$.

Using this intuition, the preparatory phase in Marker's strategy can be replaced by a short initial play in the cops and robber game. Indeed, assume that $G = G_{\mathcal P}$ for some painting $\mathcal P$, and that after finitely many steps in the cops and robber game, we have cops $c_i$ guarding a walk $W_i$ for $1 \leq i \leq k$. To each $W_i$ we can associate a cutline $f_i$ that induces it. We can further make sure that if $W_i$ and $W_j$ have an endpoint in common, then $\partial f_i$ and $\partial f_j$ intersect, i.e.\ the corresponding cutlines start or end in a common point. If the images of the cutlines are otherwise disjoint, then there is a play in the topological Marker-Cutter game where in round $i$ Marker picks $\partial f_i$ and Cutter picks $f_i$.

Similarly to Theorem \ref{t:strat}, we can show that if $\gamma = (D, \chi, g)$ is a game state obtained by the play described in the above paragraph, then $c(G)$ is at most one more than the maximum $v(D_k)$ that (restricted) Cutter can guarantee if the game starts at game state $\gamma_0 = \gamma$ (unless the initial strategy uses more cops, in which case $c(G)$ is the number of cops used in the initial strategy). The advantage of this is that it allows us to enter the potential bounding phase at a lower potential. To this end, we use a strategy which essentially corresponds to the first part of the proof of Proposition 3.2 in \cite{S01}.

\begin{lemma}
\label{l:initialplay}
Using a play as described above with at most $4$ cops, we can reach $\gamma = (D, \chi, g)$ where $D$ consists of a single cycle labelled $(a,b,a,c)$ and $g = g_0-1$. In particular, labels $b$ and $c$ are uniquely appearing.
\end{lemma}

\begin{proof}
As above assume that $G = {G_\mathcal P}$. Call a closed walk $C$ non-trivial if for some (equivalently any) cutline  $f$ with $f(0) = f(1)$ which induces $C$, we have that $\cut(F_{\mathcal P},f)$ is connected. Let $A$ be a shortest non-trivial closed walk, and assume that $A = (v_1,v_2,\dots,v_k,v_1)$.

Recall that $v_1$ is a disk in $\mathcal P$ by definition of $G_{\mathcal P}$ and let $f$ be a cutline inducing $C$ with $f(0) = f(1)$ in the interior of $v_1$. Using minimality of $A$ it is not hard to show that the image of $f$ intersects $v_1$ in a line segment, and consequently $\cut(v_1,f)$ consists of exactly two disks which we denote by $w_1$ and $w_2$. Let $B'$ be a shortest walk connecting $w_1$ and $w_2$ in $G_{\cut(\mathcal P,f)}$ and let $B$ be the image of $B'$ under the homomorphism given by Lemma \ref{lem:cutpainting-homomorphism}. 

Let $g$ be a cutline inducing $B$ such that $g(0) = g(1) = f(0) = f(1)$, and such that the images of $f$ and $g$ are otherwise disjoint. Let $f'$ be the cutline in $\cut(\mathcal P,g)$ defined by $f'(x) = f(x)$ for $x\in (0,1)$ and choosing $f'(0)$ and $f'(1)$ such that $f'$ becomes continuous. Then $f'$ induces a walk $A'$ in $G_{\cut(\mathcal P,g)}$. We claim that $A'$ is a shortest walk connecting its endpoints. If not, let $\hat A'$ be a shorter such walk and let $\hat f'$ be a cutline inducing $\hat A'$ such that $\hat f'(0) = f'(0)$ and $\hat f'(1) = f'(1)$. Let $\hat f$ be the image of $\hat f'$ under the pasting map (see Definition \ref{defn:pasting}). Since the image of $g$ meets both sides of $\hat f$ it is clear that $\cut(F_{\mathcal P},\hat f)$ is connected, and consequently the image of $\hat A$ under the homomorphism given by Lemma~\ref{lem:cutpainting-homomorphism} is a non-trivial closed walk contradicting the minimality of $A$.

Let $R$ consist of all vertices that are not contained in $A$ or $B$. Since $A'$ is a shortest walk between its endpoints in $\cut(\mathcal P,g)$, by Lemma \ref{lem:relativeguard} we conclude that $A$ is guardable in the subgraph $G_A$ of $G$ induced by $R \cup A$. Similarly, $B$ is guardable in the subgraph $G_B$ of $G$ induced by $R\cup B$.

For the strategy, label the four cops by $c_a,c_b,c_c, c_d$. Then $P_1 = (v_1,\dots,v_{\lfloor k/2 \rfloor})$ and  $P_2 = (v_{\lceil k/2 \rceil},\dots, v_k)$ must be shortest paths between their endpoints, otherwise we could find a shorter non-trivial walk $A$. In particular, by Lemma \ref{lem:guardshortestpath} we can use $c_c$ and $c_d$ to guard $P_1$ and $P_2$ in $G$ respectively. Next use cop $c_b$ to guard $B$ in $G_B$ (note that the robber must stay in $G_B$ to avoid capture). Then use $c_a$ to guard $A$ in $G_A$. From this point on, if the robber enters $A$, he is caught by $c_a$, and if he enters $B$, then he is caught by $c_b$. We may thus release $c_c$ and $c_d$ from their respective guarding duties, and reassign $c_c$ to guard $B$ together with $c_b$.

Choosing cutlines $f$,$g$, and an appropriately chosen second cutline inducing $B$, it is not hard to see that we reach the claimed game state.
\end{proof}

\begin{proof}[Proof of Theorem \ref{t:main}]
We will show that, starting from the configuration given in Lemma \ref{l:initialplay}, we can either maintain $v(\gamma) \leq \frac 43 g_0 + \frac 73$ throughout the game, in which case we can apply Theorem \ref{t:strat} to deduce the claimed bound, or we can achieve $v(\gamma) \leq \frac 43 g_0 - \frac 13$ with $g = 1$. In this case, we can switch back to the cops and robbers game and use $3$ additional cops to catch the robber.

So, let assume that we started the game according to Lemma \ref{l:initialplay}, and let the unique cycle in the resulting game state be labelled by $(a_0,b_0,a_0,c_0)$. In the first move, Marker chooses the endpoints of the edge labelled by $c_0$. By Lemma \ref{l:nochoice}, Cutter must respond with a move of type (\ref{movea}), and we arrive at a game state $\gamma = (D, \chi, g)$ where $g = g_0-2$, and $D$ consists of a cycle of length $2$ with labels $(c_0, d_0)$, and a cycle of length $4$ with labels $(a_0,b_0,a_0,c_0)$. 

We now follow the strategy given in  Figure \ref{fig_Markerstrategy}, pretending that the path with labels $(a_0,c_0,a_0)$ was an edge with uniquely appearing label $\hat a_0$. Note that this is possible, because the path contains the isolated label $a_0$ and thus can play the role of an edge with a uniquely appearing label in all applications of Lemmas \ref{l:limitedchoice} and \ref{l:nochoice}. Furthermore, the path has the same contribution to the potential as an edge with a uniquely appearing label whence Lemmas \ref{lem:movedbound}, \ref{lem:moveabound}, and \ref{lem:movebcbound} (if we allow the path with labels $(a_0,c_0,a_0)$ as a nesting path) can be used to show that applying the strategy given by Figure \ref{fig_Markerstrategy} does not decrease the potential. Finally observe that the game state $\gamma$ corresponds to situation \fbox{1} of Figure  \ref{fig_Markerstrategy} (with the unique active cycle labelled with labels $(\hat a_0, b_0)$).

Note that the potential of the game state described in Lemma \ref{l:initialplay} is $4-9+2=-3$. Thus, as long as $g \geq 1$ we have $-3 \leq g_0 - 4 - 3 v(\gamma) + 5$, and consequently $v(\gamma) \leq \frac 43 g_0 + \frac 43$ by similar computations as in Theorem \ref{t:PC}. Furthermore, if $g=0$ and 
\[
\sum_{\substack{C \in \text{Comp}(D)\\p(C, \gamma) > 0}} p(C, \gamma) \leq 4,
\]
then we obtain  $v(D) \leq \frac 43 g_0 + \frac 73$ as claimed.

We hence only need to rule out that the sum is equal to $5$ while $g=0$, in other words, we need to avoid arriving at \fbox{5} or \fbox{9} of Figure \ref{fig_Markerstrategy} with $g = 0$. Observe that this can only happen, if a few steps earlier we are at \fbox{2} with $g\in \{1,4\}$.

If we ever encounter situation \fbox{2} with $g  = 4$, we proceed as follows. Since all cycles in $D$ have potential $0$, we know that the label $\hat a_0$ appears on a cycle of length $2$ whose only other label ($x$, say) is not uniquely appearing. In other words, there is a cycle with labels $(a_0,c_0,a_0,x)$, and pretending that the path with labels $(a_0,x,a_0)$ is an edge with uniquely appearing label $\tilde a_0$ we arrive at \fbox{2} with active cycles labelled $(\tilde a_0, c_0)$ and $(c_0,d_0)$. In particular, the nesting path is now a single edge with uniquely appearing label. Thus, if Marker keeps playing the strategy given by Figure \ref{fig_Markerstrategy}, then in situation \fbox{4} Cutter cannot choose a move of type (\ref{moveb}) or (\ref{movec}), hence we arrive at \fbox{1} with $g=2$. If Cutter's next reply is (\ref{movea}), then we arrive at \fbox{2} with $g  = 1$, and consequently we can use the strategy below. If it is of type (\ref{moveb}) or (\ref{movec}), then by sticking to the strategy in Figure \ref{fig_Markerstrategy} we finish the game at \fbox{12}.

If we encounter situation \fbox{2} with $g  = 1$, we switch back to the cops and robber game. Note that at this point all cycles have potential $0$, and similar calculations as in Theorem \ref{t:PC} give $v(\gamma) \leq \frac 43 g_0 - \frac 13$. In other words, in the cops and robber game we have at most $\frac 43 g_0 - \frac 13$ cops guarding some subgraphs of $G$, and since $g=1$ we know that the graph spanned by the vertices that the robber can still visit without being caught has genus at most $1$. In particular, by results from \cite{L19}, we can catch the robber using at most $3$ additional cops thus proving Theorem \ref{t:main}.
\end{proof}

\bibliographystyle{plain}
\bibliography{cop}

\end{document}